\documentclass[a4paper,12pt]{amsart}
\usepackage[T1]{fontenc}
\usepackage[utf8]{inputenc}
\usepackage[margin=2.5cm]{geometry}

\usepackage{amsmath, amssymb,amsfonts}
\usepackage{thm-restate}
\usepackage{array}
\usepackage{multirow}
\usepackage[dvipsnames]{xcolor}
\usepackage{tikz}
\usetikzlibrary{shapes.misc}
\usetikzlibrary{shapes.geometric}
\usetikzlibrary {decorations.pathmorphing}

\tikzset{unseulcoin fill/.style={append after command={
   \pgfextra
        \draw[sharp corners, fill=#1, color=#1, line width = 0mm]%
    (\tikzlastnode.west)%
    [rounded corners=0pt] |- (\tikzlastnode.north)%
    [rounded corners=0pt] -| (\tikzlastnode.east)%
    [rounded corners=0pt] |- (\tikzlastnode.south)%
    [rounded corners=5pt] -| (\tikzlastnode.west);
   \endpgfextra}}}

\tikzset{black node/.style = {draw=black,fill=black,circle, inner sep=0, minimum size=0.15cm}}
\usepackage{multicol}

\usepackage{hyperref}
\hypersetup{
  pdftitle = {Coarse cops and robber in graphs and groups},
  pdfauthor = {L. Esperet, H. Gahlawat, U. Giocanti},
  colorlinks = true,
  linkcolor = black!30!red,
  citecolor = black!30!green
}
  \usepackage[capitalise, compress, nameinlink, noabbrev]{cleveref}
\crefname{rule}{}{}
\creflabelformat{rule}{#2(R#1)#3}

\usepackage{paralist}

\newtheorem{theorem}{Theorem}[section]
\newtheorem{corollary}[theorem]{Corollary}

\newtheorem{proposition}[theorem]{Proposition}

\newtheorem{conjecture}[theorem]{Conjecture}

\newtheorem{claim}[theorem]{Claim}
\newtheorem{question}{Question}

\def\cqedsymbol{\ifmmode$\lrcorner$\else{\unskip\nobreak\hfil
\penalty50\hskip1em\null\nobreak\hfil$\lrcorner$
\parfillskip=0pt\finalhyphendemerits=0\endgraf}\fi}

\newcommand*\sg[1]{\{ #1 \}}

\newcommand{\scop}{s_{c}}
\newcommand{\sr}{s_{r}}

\DeclareMathOperator{\tw}{\mathrm{tw}}

\DeclareMathOperator{\wco}{{\sf wCop}}
\DeclareMathOperator{\sco}{{\sf sCop}}
\DeclareMathOperator{\wdi}{{\sf wDiv}}
\DeclareMathOperator{\sdi}{{\sf sDiv}}

\newenvironment{proofofclaim}{%
  \proof}{\endproof}

\let\le\leqslant
\let\ge\geqslant
\let\leq\leqslant
\let\geq\geqslant

\title[Coarse cops and robber in graphs and groups]{Coarse cops and robber in graphs and groups}

\author[L.~Esperet]{Louis Esperet}
\address[L.~Esperet]{Univ.\ Grenoble Alpes, CNRS, Laboratoire G-SCOP,
  Grenoble, France}
\email{louis.esperet@grenoble-inp.fr}

\author[H.~Gahlawat]{Harmender Gahlawat}
\address[H.~Gahlawat]{LIMOS, Université Clermont Auvergne, Clermont-Ferrand, France}
\email{harmendergahlawat@gmail.com}

\author[U.~Giocanti]{Ugo Giocanti}
\address[U.~Giocanti]{Faculty of Mathematics
and Computer Science, Jagiellonian University, Krak\'ow, Poland}
\email{ugo.giocanti@uj.edu.pl}

\date{\today}

\thanks{The authors are partially supported by the French ANR Projects TWIN-WIDTH
  (ANR-21-CE48-0014-01), and by LabEx
  PERSYVAL-lab (ANR-11-LABX-0025). The third author is supported by the National Science Center of Poland
under grant 2022/47/B/ST6/02837 within the OPUS 24 program}

\begin{document}

\maketitle
\begin{abstract}
We investigate two variants of the classical Cops and robber game in graphs, recently introduced by Lee,  Mart\'inez-Pedroza, and Rodr\'iguez-Quinche. The two versions are played in infinite graphs and the goal of the cops is to prevent the robber to visit some ball of finite radius (chosen by the robber) infinitely many times. Moreover the cops and the robber move at a different speed, and the cops can choose a radius of capture before the game starts. Depending on the order in which the parameters are chosen, this naturally defines two games, a weak version and a strong version (in which the cops are more powerful), and thus two variants of the cop number of  a graph $G$: the weak cop number $\wco(G)$ and the strong cop number $\sco(G)$. It turns out that these two parameters are invariant under quasi-isometry and thus we can investigate these parameters in finitely generated groups by considering any of their Cayley graphs; the parameters do not depend on the chosen set of generators.

We answer a number of questions raised by Lee,  Mart\'inez-Pedroza, and Rodr\'iguez-Quinche, and more recently by Cornect and Mart\'inez-Pedroza. 
\begin{itemize}
    \item We show that if some graph $G$ has a quasi-isometric embedding in some graph $H$, then $\wco(G)\le \wco(H)$ and $\sco(G)\le \sco(H)$.
    \item It was proved by Lee,  Mart\'inez-Pedroza, and Rodr\'iguez-Quinche that Gromov-hyperbolic graphs have strong cop number equal to 1. We prove that the converse also holds, so that $\sco(G)=1$ if and only if $G$ is Gromov-hyperbolic. This gives a new purely game-theoretic definition of hyperbolicity in infinite graphs.
    \item We tie the weak cop number of a graph $G$ to the existence of asymptotic minors of large treewidth in $G$. We deduce that for any graph $G$, $\wco(G)=1$ if and only if $G$ is quasi-isometric to a tree. In particular, for any finitely generated group $\Gamma$, $\wco(\Gamma)=1$ if and only if $\Gamma$ is virtually free. We also prove that for any finitely presented group $\Gamma$, $\wco(\Gamma)=1 \text{ or }\infty$.
    \item We prove that $\sco(\mathbb{Z}^2)=\infty$ (this was only known to hold for the weak version of the game).
\end{itemize}

\medskip

We have learned very recently that some of our results have been obtained independently by Appenzeller and Klinge, using fairly different arguments.
\end{abstract}

\section{Introduction}\label{sec:intro}
 
We consider two recent variants of the classical cops and robber game in graphs originally introduced by Quilliot~\cite{qui}, and  Nowakowski and Winkler~\cite{nowakowski}. At a high level, the setting of the new versions of the cops and robber game introduced by Lee,  Mart\'inez-Pedroza, and Rodr\'iguez-Quinche \cite{lee2023coarse} is the following: the games are played in some infinite graph $G$; the cops and a robber move in the graph $G$ in alternating rounds, with the cops moving at speed $s_c$ and the robber moving at speed $s_r$ (meaning that in a single round, each cop can move from a vertex $u$ to a vertex $v$ if the distance $d_G(u,v)$ between $u$ and $v$ in $G$ is at most $s_c$, and similarly for the robber). The cops have some radius of capture $\rho$, which means that if at any point of the game a cop is within distance $\rho$ from the robber, then the robber is captured and the game stops. The goal of the cops is to make sure that after finitely many moves, the robber never enters again in some fixed ball of radius $R$ (chosen by the robber). All parameters $s_c, s_r, \rho, R$ are finite. 

There are two natural variants of the game depending in which order the parameters are chosen by the players:
\begin{itemize}
\item the strong game: $s_c, s_r, \rho, R$, and
\item the weak game: $s_c, \rho, s_r, R$.
\end{itemize}

Since the robber can choose his own speed after the cops have specified their speed and radius of capture in the weak game, the robber is more powerful in the latter game (and thus the cops are less powerful). For each game, the cop number is defined as the infimum number of cops needed to win the game if the parameters are chosen in the corresponding order by the players. For the strong game, this defines the \emph{strong cop number} $\sco(G)$ of a graph $G$ and for the weak game, this defines the \emph{weak cop number} $\wco(G)$ of  $G$. Note that by definition we have $\sco(G)\le \wco(G)$ for any graph $G$.

\subsection*{Quasi-isometry} A crucial property shared by these two parameters is that they are invariant under quasi-isometry, as proved in \cite{lee2023coarse}. As all Cayley graphs of a given finitely generated group $\Gamma$ are quasi-isometric, we can define the strong cop number $\sco(\Gamma)$ and the  weak cop number $\wco(\Gamma)$ of $\Gamma$ as the strong and weak cop numbers of any of their Cayley graphs (this is well defined). The way the invariance under quasi-isometry is proved in \cite{lee2023coarse} uses the notion of \emph{quasi-retract}. More precisely, it was shown that if $G$ and $H$ are connected graphs and $G$ is a quasi-retraction of $H$ then $\wco(G)\le \wco(H)$ and $\sco(G)\le \sco(H)$. If two connected graphs are quasi-isometric, there exist quasi-retractions from the first to the second and vice-versa, and it thus follows that the strong and weak cop numbers of connected graphs are invariant under quasi-isometry. A natural question, raised in \cite{Cornect_Martinez24}, asks whether this can be proved more directly through the notion of quasi-isometric embedding. In other words, are the strong and weak cop numbers monotone under quasi-isometric embeddings? Our first result is a positive answer to this question.

\begin{theorem}\label{thm:qi-intro}
If $G$ and $H$ are two connected graphs and $G$ has a quasi-isometric embedding into $H$, then $\wco(G)\le \wco(H)$ and $\sco(G)\le \sco(H)$.
\end{theorem}

As for any two  quasi-isometric graphs, the first has a quasi-isometric embedding into the second and vice-versa, we obtain an alternative proof of the invariance of the strong and weak cop numbers under quasi-isometry (but we note that directly proving this invariance under quasi-isometry is much simpler than proving Theorem \ref{thm:qi-intro}).

\subsection*{Hyperbolicity}

The notion of \emph{hyperbolicity} for metric spaces was introduced by  Gromov 
\cite{Gromov_hyperbolic}. Hyperbolicity is a quasi-isometric invariant, and a crucial property of hyperbolic groups is that they have solvable word problem. Lee,  Mart\'inez-Pedroza, and Rodr\'iguez-Quinche \cite[Theorem A]{lee2023coarse} proved that every hyperbolic graph has strong cop number 1. Our second result is that this is indeed an equivalence.

\begin{theorem}\label{thm:hy}
For any connected graph $G$, $\sco(G)=1$ if and only if $G$ is hyperbolic. In particular, for any   finitely generated group $\Gamma$, $\sco(\Gamma)=1$ if and only if $\Gamma$ is hyperbolic.
\end{theorem}

This seems to be the first game-theoretic characterization of hyperbolicity in infinite graphs or groups. 
Note that \emph{finite} hyperbolic graphs were already proved to admit a characterisation by means of a Cops and robber game in \cite{CCPP14}. In this game, one cop with speed $s_c$ (and reach $0$) must capture a robber, and the robber is allowed to choose any speed $s_r$  once $s_c$ has been chosen by the cop. The following was proved in \cite{CCNV, CCPP14}: 

\begin{enumerate}
    \item every finite $\delta$-hyperbolic graph is cop-win when the speeds $(s_c,s_r)$ are such that $s_r-s_c < 2\delta$, and
    \item every finite graph which is cop-win for some pair of speeds $(s_r, s_c)$ such that $s_c<s_r$ is $\delta$-hyperbolic, with $\delta=O(s_r^2)$.
\end{enumerate}
However, note that if the graph we consider admits an infinite geodesic path, then the robber can always evade the cop by choosing a speed which is greater than the speed of the cop and traveling on this path. For infinite graphs, is was proved in \cite{CCPP14} that hyperbolic graphs can be characterized by a certain dismantlability condition, which \emph{for finite graphs} turns out to characterise exactly the cop-win graphs for this game \cite{CCNV}. Note however that this characterization is structural, and does not provide a game-theoretic definition of hyperbolicity in infinite graphs.

\subsection*{Weak cop number 1} We have seen above that having strong cop number equal to 1 is equivalent to being hyperbolic. A natural problem is to find a similar characterization of graphs or groups of weak cop number equal to 1. Since trees have weak cop number 1 and the weak cop number is invariant under quasi-isometry, it follows that any graph which is quasi-isometric to a tree has weak cop number equal to 1. We prove that the converse also holds. In particular, in the specific case of groups, our result provides a positive answer to Question 1.4 in \cite{Cornect_Martinez24}.

\begin{theorem}\label{thm:wcop1}
For any connected graph $G$, $\wco(G)=1$ if and only if $G$ is quasi-isometric to a tree.  In particular, for any   finitely generated group $\Gamma$, $\wco(\Gamma)=1$ if and only if $\Gamma$ is virtually free.
\end{theorem}

It was observed by Lee,  Mart\'inez-Pedroza, and Rodr\'iguez-Quinche that while they have examples of graphs of weak cop number or strong cop number equal to $n$ for any integer $n$, for groups they only have examples whose weak (or strong) cop number is equal to 1 or $\infty$. They raised the following question \cite[Question K]{lee2023coarse} (see also \cite[Question 1.3]{Cornect_Martinez24}).

\begin{question}\label{q:1infty}
Is it true that for any finitely generated group $\Gamma$, we  have $\wco(\Gamma)=1\text{ or }\infty$, and $\sco(\Gamma)=1\text{ or }\infty$?
\end{question}

More generally, they asked whether for any connected locally finite vertex-transitive graph $G$ we always have $\wco(G)=1\text{ or }\infty$, and $\sco(G)=1\text{ or }\infty$ \cite[Question L]{lee2023coarse}. We make progress on this question by relating the weak cop number of a graph with the presence of so-called \emph{asymptotic minors} of large treewidth.

\begin{theorem}
 \label{thm: minor intro}
If a graph $G$ contains a finite graph $H$ as an asymptotic minor, then
 $\wco(G)\geq \tw(H)$.
\end{theorem}

Using this result, we connect Question \ref{q:1infty} to a recent (disproved) conjecture of Georgakopoulos and Papasoglu on graphs with no large grids as asymptotic minors \cite{GP23}. Using recent partial positive results of Albrechtsen and Hamann \cite{AH24} on this conjecture, we deduce the following partial answer to Question \ref{q:1infty}.

\begin{theorem}\label{thm:1infty}
For any finitely presented group $\Gamma$, $\wco(\Gamma)=1\text{ or }\infty$.
\end{theorem}

Using Theorem \ref{thm:wcop1}, this immediately implies that any finitely presented group which is not virtually free has infinite weak cop number (generalizing a number of results of \cite{lee2023coarse,Cornect_Martinez24} on specific groups). 

\subsection*{The 2-dimensional grid} The final result is about the strong cop number of the Euclidean plane (or equivalently by quasi-isometry, the 2-dimensional square grid). It was proved in \cite[Theorem C]{lee2023coarse} that the 2-dimensional grid has infinite weak cop number. The authors asked whether this extends to the strong cop number \cite[Question D]{lee2023coarse}. We prove that this is indeed the case.

\begin{theorem}\label{thm:z2}
$\sco(\mathbb{Z}^2)=\infty$.
\end{theorem}

It was proved in \cite{Cornect_Martinez24} that Thompson's group $F$ has infinite weak cop number, using the result that $\wco(\mathbb{Z}^2)=\infty$ \cite{lee2023coarse} and the fact that $F$ retracts onto $\mathbb{Z}^2$. Using Theorem \ref{thm:z2}, the same proof immediately yields the stronger statement that $\sco(F)=\infty$.

\subsection*{Divergence games} 
 Our lower bounds on the weak and strong cop numbers are proved by giving winning strategies for the robber (in which he visits   a given ball of finite radius infinitely many times). It turns out that in most of our proofs, the robber can in fact do much more: he can simply remain inside a given ball of finite radius, forever evading the cops. This motivates us to introduce a variant of the weak and strong games studied in this paper where the goal of the cops is to make sure that the distance between the robber and some root vertex $v$ tends to infinity. We conclude by stating a number of observations on this game, and asking some natural questions about it: in particular, are the corresponding cop numbers different than in the original setting of \cite{lee2023coarse}?

\subsection*{Organization of the paper} We start with some definitions and notation in Section \ref{sec:prel}. In Section \ref{sec:qi} we prove that the weak and strong cop number are monotone under quasi-isometric embeddings. In Section \ref{sec:hy}, we prove that any graph with strong cop number 1 is hyperbolic. In Section \ref{sec:am} we highlight the connections between weak cop numbers and asymptotic minors of large treewidth, and give a number of applications. In Section \ref{sec:z2} we prove that $\sco(\mathbb{Z}^2)=\infty$. We conclude in Section \ref{sec:div} by introducing and studying the Divergence games mentioned in the paragraph above, and an alternative version of the weak and strong games studied in the paper where the robber needs to come back infinitely often in some finite subgraph (rather than a subgraph of finite radius).

\subsection*{Related work} We have learned very recently that a number of our results have been obtained independently by Appenzeller and  Klinge  \cite{AK25}. This includes the results on graphs of weak cop number 1, graphs of strong cop number 1, and the strong cop number of the grid. The proofs of all their results appear to be  quite different from our proofs, using different characterizations of hyperbolicity and quasi-trees.

\medskip

Our Theorem \ref{thm:1infty} (also appearing as Corollary \ref{cor: weak-infty} later in the paper), stating that the weak cop number of every finitely presented group is either 1 or $\infty$, has been subsequently extended by Lehner \cite{Leh25} to all locally finite vertex-transitive graphs. (We note that our results relating weak cop numbers and the treewidth of asymptotic minors (Theorem \ref{thm: minor intro}) hold for any connected graph, not necessarily vertex-transitive or even locally finite.)

\section{Preliminaries}\label{sec:prel}

 \subsection{Graph Theory}
For a graph $G$, we denote its vertex set by $V(G)$ and edge set by $E(G)$. If $G$ is finite, we denote the size of $V(G)$ by $n$ and size of $E(G)$ by $m$. In this paper, we consider only 
simple graphs. For a graph  $G$  and a subset $U\subseteq V(G)$, $G[U]$ denotes the subgraph of $G$ induced by $U$.

Let $v$ be a vertex of a graph $G$. Then, by $N(v)$ we denote the \textit{open neighbourhood} of $v$, that is, $N(v)= \{u : uv \in E(G)\}$. 
By $N[v]$ we denote the \textit{closed neighbourhood} of $v$, that is, $N[v] = N(v) \cup \{v\}$. For $X \subseteq V(G)$, we define $N_X(v) = N(v) \cap X$ and $N_X[v] = N[v] \cap X$. We say that $v$ \textit{dominates} $u$ if $u\in N[v]$. 

A {\em $u,v$-path} $P$ is a path with endpoints $u$ and $v$, and the \emph{length} of a path $P$ is the number of edges on $P$. A path is a \textit{geodesic} if it is a shortest path between its endpoints. For $u,v\in V(G)$, let $d_G(u,v)$ denote the length of a shortest $u,v$-path. As all graphs considered in the paper are connected, $d_G(u,v)$ is well-defined (and finite). 

For a vertex $v$ in a graph $G$, and an integer $t\ge 0$, the \emph{ball of radius $t$ centered in  $v$}, denoted by $N^t_G[v]$, is the set of vertices at distance at most $t$ from $v$ in $G$. For a set $X\subseteq V(G)$, we write $N^t_G[X]$ for the set of vertices at distance at most $t$ from $X$ in $G$.

When the graph $G$ under consideration is clear from the context, we drop the subscript $G$ in $N_G^t[X]$ and $d_G(u,v)$ to ease the presentation. 

\medskip

A \emph{tree-decomposition} of a graph $G$ is a pair $(T,\mathcal V)$, where $T$ is a tree and $\mathcal V=(V_t)_{t\in V(T)}$ is a family of subsets $V_t$ of $V(G)$ such that:
\begin{itemize}
 \item $V(G)=\bigcup_{t\in V(T)}V_t$;
 \item for every nodes $t,t',t''$ such that $t'$ is on the unique path of $T$ from $t$ to $t''$, $V_t\cap V_{t''}\subseteq V_{t'}$; 
 \item every edge $e\in E(G)$ is contained in some induced subgraph $G[V_t]$ for some $t\in V(T)$. 
\end{itemize}

The sets $V_t$ with $t\in V(T)$ are called the \emph{bags} of $(T,\mathcal V)$.
The \emph{width} of $(T,\mathcal V)$ is the supremum of $|V_t|-1$, for $t\in V(T)$ (which might be infinite), and the \emph{treewidth} of a graph $G$, denoted by $\mathrm{tw}(G)$, is the infimum of the width of $(T,\mathcal V)$, 
among all tree-decompositions $(T,\mathcal V)$ of $G$. For every edge $e=uv\in E(T)$, if we let $T_1,T_2$ denote the two connected components of $T-e$ and for every $i\in \sg{1,2}$, we set $V_i:=\bigcup_{t\in V(T_i)}V_{t}$, then $(V_1, V_2)$ and $(V_2, V_1)$ are two separations of $G$, which we call the \emph{edge-separations} associated to $e$. 
 
\subsection{Cops and robber}
We consider two variants of the Cops and robber game, both of which are perfect information games. The game is played in a fixed graph $G$ between two players, a \textit{cop player} controlling $k$ cops and a \textit{robber player} controlling a single \textit{robber}, and has the following parameters of interest (all of which are non-negative integers), and the two variants of the game depend on the order in which these parameters are chosen:
\begin{itemize}
    \item Speed of cops, denoted by $\scop$. 
    \item Speed of the robber, denoted by $\sr$.
    \item Reach of the cops, denoted by $\rho$.
    \item Radius of the ball to protect, denoted by $R$.
\end{itemize}

The game is played as follows. Initially, the $k$ cops choose vertices $c^1_0,\ldots ,c^k_0$ from $G$ as their initial positions. Then, knowing the initial positions of the cops, the robber chooses a vertex $r_0$ as its initial position. Then the cops and the robber move alternately with the cops moving first. A move of the cops followed by a move of the robber is called a \emph{stage}. The vertices where the cops and robber are located at the end of the $i$-stage are denoted by $c^1_
i,\ldots ,c^k_i$ and $r_i$, respectively.
At the beginning of the $i$-stage, each cop can move from its current position to any vertex at a distance at most $\scop$, that is,
$d(c^j_{i-1}, c^j_{i})\leq \scop$ for every $j\in\{1,\ldots ,k\}$. The robber is \emph{captured} during the $i$-stage if $d(r_{i-1}, c^j_{i})\leq \rho$ for some $j\in\{1,\ldots ,k\}$. After the cops have moved, if the robber has not been captured, the robber can move from its current position $r_{i-1}$ to a vertex $r_i$ if there is a path from $r_{i-1}$ to $r_i$ of length at most $\sr$ such that every vertex in the path is at distance larger than $\rho$ from any of the cop positions $c^1_i, \ldots, c^k_i$.  The cops win the game with parameters $s_c,s_r,\rho,R$, if, for any vertex $v\in V(G)$, eventually they can protect the $R$-ball centered at $v$, by which we mean that the robber is captured or there is $N>0$ such that $d(v,r_i)>R$ for all $i\geq N$. We emphasize that the center $v$ of the ball protected by the cops is known to all players (including the robber) before the other parameters are chosen (including the initial positions of the cops and the robber).

\subsection{Definition of the weak and strong cop numbers.}
We say that a graph $G$ is \emph{$\mathsf{CopWin}(k,\scop,\rho,\sr, R)$} if for any vertex $v$ of $G$, $k$ cops with speed $\scop$ and reach $\rho$ can eventually protect the $R$-ball centered at $v$.
The definitions of the weak cop number and the strong cop number differ in the order in which the parameters of the game are chosen by the two players. For the weak cop number, the robber has an information advantage: the cops choose their speed $\scop$ and reach $\rho$ and then the robber, knowing this information, chooses his speed $\sr$ and the radius $R$ of the ball to protect. For the strong cop number the cops have the advantage of choosing their reach $\rho$ after knowing the robber's speed. More precisely, these invariants are defined as follows:
\begin{itemize}
\item 	We say that $G$ is \emph{$k$-weak cop win} if there exists $\scop \in \mathbb{N}_{>0}$ and $\rho \in \mathbb{N}_{\geq 0}$ such that for any $\sr, R \in \mathbb{N}_{>0}$, $\Gamma$ is $\mathsf{CopWin}(k,\scop,\rho,\sr , R)$. 
\item	
We say that $G$ is \emph{$k$-strong cop win} if there exists $\scop \in \mathbb{Z}_{>0}$ such that for any $\sr \in \mathbb{Z}_{>0}$, there is $\rho \in \mathbb{Z}_{\geq 0}$ such that for any $R \in \mathbb{Z}_{>0}$, $G$ is $\mathsf{CopWin}(k,\scop,\rho,\sr , R)$. 
\end{itemize}
The \emph{strong cop number $\sco(G)$} of a graph $G$ is defined as the
infimum value of $k$ such that $G$ is $k$-strong cop win (with $\sco(G)=\infty$ if there is no such $k$).
The \emph{weak cop number $\wco(G)$} is defined analogously as the infimum $k$ such that $G$ is $k$-weak cop win.
Observe that
\[ 0< \sco(G) \leq \wco(G) .\]
For example, the infinite 2-way path has weak and strong cop number one, as a single cop can push the robber away from any ball. 

\medskip

A simple yet important remark is that the weak (resp.\ strong) cop number of a graph is the supremum of the weak (resp.\ strong) cop numbers of its connected components.  So in the paper we can restrict our attention to connected graphs, and all results can be naturally extended to general graphs by taking the supremum over all connected components.

\medskip

 We observe that when $G$ is connected, the choice of $v$, the center of the ball, is irrelevant in both  games. In particular, if the robber has a winning strategy against $k$ cops when they have to protect a ball centered at $v$ with game parameters $(k,\scop,\rho,\sr , R)$, then  the robber has a winning strategy against $k$ cops when they have to protect a ball centered at any $v'\in V(G)$ with game parameters $(k,\scop,\rho,\sr , R+d(v',v))$. Since the robber chooses the radius of the ball in the end, he can choose the radius to be $R+d(v',v)$. Hence, the robber as a winning strategy when the center of the ball is $v$ if and only if he has a winning strategy when the center of the ball is $v'$. Thus we can always  assume without loss of generality that the robber chooses $v$, the center of the ball to be protected, at the same time as choosing the radius of the ball.

\section{Quasi-isometric embeddings}\label{sec:qi}

Let  $G$ and $H$ be two graphs. We say that $G$ is
\emph{quasi-isometric} to $H$ if there is a map $h: V(G) \rightarrow V(H)$
and constants $C\ge 1$, and $D\ge 0$ such that
\begin{enumerate}
    \item for any $u\in V(H)$ there is $x\in V(G)$ such that $d_H(u,h(x))\le D$, and
    \item for every $x,y\in V(G)$, $$\frac1{C}\, d_G(x,y)-C\le d_H(h(x),h(y))\le C \,d_G(x,y)+C.$$
\end{enumerate}
It is not difficult to check that quasi-isometry is an equivalence relation on graphs, and we
often simply say that $G$ and $H$ are quasi-isometric. If 
condition (1) is omitted in the definition above, we say that $f$ is a
\emph{quasi-isometric embedding of $G$ into $H$}. Note that if $G$ and $H$ are quasi-isometric then $G$ has a quasi-isometric embedding into $H$ and $H$ has a quasi-isometric embedding into $G$, but the converse does not necessarily hold. 

\medskip

The main result of this section is that the weak and strong cop numbers are monotone under quasi-isometric embedding.

\begin{theorem}\label{thm:qi}
Assume that a connected graph $G$ has a quasi-isometric embedding into a connected graph $H$. Then $\wco(G)\le \wco(H)$ and $\sco(G)\le \sco(H)$.
\end{theorem}

\begin{proof}
  Assume that there is a map $h: V(G) \rightarrow V(H)$ and a constant $C\ge 1$ such that  for every $x,y\in V(G)$, $\frac1{C}\, d_G(x,y)-C\le d_H(h(x),h(y))\le C \,d_G(x,y)+C$. Let us denote the speed of the cops in $H$ by $s\ge 1$, and their radius of capture by $\rho\ge 0$.   

  \smallskip

  We will prove that the robber can use any winning strategy in $G$ against $t$ cops to devise a winning strategy in $H$ against $t$ cops. Let $U=h(V(G))$ be the image of $V(G)$ in $V(H)$. If all cops in $H$ are close to $U$, then the robber can map their positions and his own position in $H$ to some virtual positions in $G$, using the preimage of the elements of $U$ under $h$. It can then use a winning strategy in $G$, move in $G$ accordingly, and then map its new position in $G$ to a new position in $H$, using $h$. The situation is less clear if some cops in $H$ are far from $U$, and the difficulty of the proof will be to deal with this issue.

  \medskip

  Let $U^+=N_H^C[U]$ be the set of vertices at distance at most $C$ from $U$ in $H$. We will need the following useful observation.

  \begin{claim}\label{C:moveShadow}
    Let $u,v \in U$ be two vertices of $H$. Then there is a path $P\subseteq U^+$ between $u$ and $v$ in $H$ which is the concatenation of at most $C\,d_H(u,v)+C^2$ paths of length $2C$ whose endpoints are all in $U$. In particular $P$ has length at most $2C^2\, d_H(u,v)+2C^3$.
    \end{claim}
    \begin{proofofclaim}
      Let $x,y \in V(G)$ be two vertices such that $h(x) = u$ and $h(y) = v$. By the definition of $h$, we have $d_G(x,y) \leq C\,d_H(u,v)+C^2$. Hence, there exists an $(x,y)$-path $Q = (x= x_0,\ldots,x_{k} = y)$ in $G$ of length $k \leq C\, d_H(u,v)+C^2$. Note that for each $i\in [k]$, there is an $(h(x_{i-1}),h(x_i))$-path in $H$ of length at most $2C$, and in particular this path is contained in
       $U^+$. The concatenation of these $k$ paths is a walk that contains a $(u,v)$-path in $H$ of length at most $2C(C\, d_H(u,v)+C^2)$ that is contained in $U^+$. This completes the proof of our claim. 
     \end{proofofclaim}

     Before mapping the positions of the cops in $H$ to some virtual positions in $G$, the robber will map the positions of the cops in $H$ to some virtual positions in $U$.    For a vertex $v\in V(H)$, let $\pi(v)$ denote the subset of vertices of $U$ at distance $d_H(v,U)$ from $v$ in $H$ (that is, these vertices of $U$ minimize the distance between $v$ and $U$).  For each cop $c$ in $H$, we create a corresponding (virtual) cop denoted by $c_U$. We will maintain the following rules for the moves of $c_U$.  

     \begin{enumerate}[(R1)]
     \item Each cop $c_U$ moves in $U^+$ at speed $s_U=500C^3s$, and lies in $U$ at the end of every move.\label[rule]{r1}
     \item If $c$ occupies a vertex $v$, then $c_U$ occupies a vertex $u\in U$ such that for any $w\in \pi(v)$, $d_H(u,w) \leq 16C^3\, d_H(v,w)$.\label[rule]{r2}
     \end{enumerate}
     
     We will show that for every move of a cop $c$ in $H$, the robber can move $c_U$ while respecting the two rules above. Moreover, assuming that each virtual cop $c_U$ has a radius of capture $\rho_U=2\rho+C+16C^3(\rho+C)$, we will show that if the robber stays at distance more than $\rho_U$ from the virtual cop $c_U$ while remaining in $U^+$, then it also stays at distance more that $\rho$ from the cop $c$ in $H$.

     \medskip

     Now, we define some \emph{safe} game states corresponding to the position of a cop $c$ in $H$ and the position of $c_U$. Let $c$ be placed on a vertex $v\in V(H)$ and $c_U$ be placed on a vertex $u\in U$ after a cop move. We have the following two safe states:

    \medskip
    \noindent\textbf{Safe State 1:} $d_H(v,U)\leq 32 s$ and $u\in \pi(v)$. 

    \smallskip
    \noindent\textbf{Safe State 2:} $d_H(v,U)>32 s$ and there exists a vertex $v'$ such that the following holds
    \begin{enumerate}[(i)]
    \item $u\in \pi(v')$
    \item there exists a $(v',v)$-path $P$ of length at most $\tfrac14\, d_H(v',u)$ such that $d_H(P,U) >32s$.
    \end{enumerate}
    
    Note in particular that if $c$ is located in $v$ and $c_U$ is located in a vertex of $\pi(v)$, then $c$ and $c_U$ are in safe state 1 or 2.

    \medskip

    We begin with the easy claim that shows that if we can reach a safe state, then \cref{r2} holds.
    \begin{claim}\label{C:SafeStates}
        If $c$ and $c_U$ are placed in $H$ in safe state 1 or 2, then \cref{r2} holds.
    \end{claim}
    \begin{proofofclaim}        
      The proof of the claim is obvious if we are in safe state 1.
      Hence, let us assume that $c$ and $c_U$ are in safe state 2. Let $v'$ be a vertex as in the definition of safe state 2. Let $w\in \pi(v)$. Then, observe that
      \begin{equation}\label{eq1}
        d_H(u,w) \leq d_H(u,v')+d_H(v',v)+d_H(v,w)\le \tfrac54 d_H(u,v')+d_H(v,w).
        \end{equation}Furthermore, since $u\in \pi(v')$, $d_H(u,v')\le d_H(w,v') \leq d_H(w,v)+d_H(v,v')$, and hence, $d_H(v,w) \geq d_H(u,v')-d_H(v',v) \ge \tfrac34d_H(u,v')$. It follows that $d_H(u,v')\le \tfrac43 d_H(v,w)$. Hence, \cref{eq1}  becomes $d_H(u,w) \leq \tfrac53 d_H(v,w)+d_H(v,w) \leq \tfrac83 d_H(v,w)$, and thus \cref{r2} holds.
    \end{proofofclaim}

    We will now prove that if $c$ and $c_U$ are not in a safe state anymore after a move of $c$, then the robber can move $c_U$ according the rules \cref{r1} and \cref{r2} and return to a safe state in a finite number of moves.

     \begin{claim}\label{C:move}
       For every move of a cop $c$ in $H$, the robber can move $c_U$ according the rules \cref{r1} and \cref{r2}. Moreover, if the robber stays at distance more than \[\rho_U=2\rho+C+16C^3(\rho+C)\] from  the cop $c_U$ in $H$, while remaining in $U^+$ the whole time, then it also stays at distance more that $\rho$ from the cop $c$ in $H$.
     \end{claim}

     \begin{proofofclaim}
       Let $v^*$ be the location of $c$ at the beginning of the game. Then the robber chooses for $c_U$ any vertex $u^*\in \pi(v^*)$. If $d_H(v^*,U)\le 32s$, then we are in safe state 1, and if $d_H(v^*,U)> 32s$ we are in safe state 2 (as seen by choosing $v'=v^*$ in the definition of safe state 2).
       
       As long as $c$ and $c_U$ remain in a safe state after a move of $c$, we do not change the location of $c_U$.

       Now we assume that $c$ and $c_U$ are in a safe state, with $c$ located on $v_0$ and $c_U$ located on $u_0\in U$, and then $c$ moves to a new location $v$, so that $c$ and $c_U$ are not in a safe state anymore after the move.

       Assume first that $d_H(v,u_0)\le 100s$. Let $u$ be any vertex in $\pi(v)$. Then, $d_H(u,v)\le 100s$, and thus $d_H(u,u_0)\le 200s$. By Claim \ref{C:moveShadow}, there is a path $P$ in $U^+$ between $u_0$ and $u$ of length at most $400C^2s+2C^3\le 500C^3s=s_U$. The robber then moves $c_U$ along $P$ in a single step to place it  on $u$. Note that if $d_H(v,u)\le 32s$ then $c$ and $c_U$ are in safe state 1, and if $d_H(v,u)> 32s$ then $c$ and $c_U$ are  in safe state 2 (with $v'=v$ in the definition of safe state 2).

       Assume now that $d_H(v,u_0)> 100s$. In particular, as $d_H(v_0,v)\leq s$, we have $d_H(v_0,u_0)> 32s$ and thus $c$ and $c_U$ were in safe state 2 when they where located on vertices $v_0$ and $u_0$. Let $v'_0$ be a vertex such that $u_0\in \pi(v_0')$ and $d_H(v_0,v_0')\le \tfrac14 d_H(u_0,v_0')$ (the existence of such a vertex follows from the definition of safe state 2). Let $\alpha=d_H(u_0,v_0')/s$. Then $d_H(u_0,v)\le \alpha s+\alpha s/4+s=(5\alpha/4+1)s$.
Note that since $(5\alpha/4+1)s\ge d_H(v,u_0)> 100s$, we have $\alpha\ge 79$.
       
Let $u\in U$ be any vertex from $\pi(v)$. Then $d_H(u_0,u)\le d_H(u_0,v)+d_H(v,u)\le  2d_H(u_0,v)\le (5\alpha/2+2)s$.  By Claim \ref{C:moveShadow}, there is a path $P$ in $U^+$ between $u_0$ and $u$ which is the concatenation of at most $C(5\alpha/2+2)s+C^2\le C^2(5\alpha/2+3)s$ paths of length at most $2C$, whose endpoints are all in $U$. In particular $P$ has length at most $2C^3(5\alpha/2+3)s$.
The robber then moves $c_U$ along $P$ to place it  on $u$. Recall that $c_U$ travels at speed $s_U=500C^3s$, so it can travel at speed at least $500C^3s-2C\ge 498C^3s$ along $P$ while always finishing his move on vertices of $U$ (to respect rule \cref{r1}). In particular 
it takes him at most \[\left\lceil \frac{2C^3(5\alpha/2+3)s}{498C^3s} \right\rceil= \left\lceil \tfrac1{249}(5\alpha/2+3)\right\rceil \le \alpha/40\]
steps to do so (here we have used $\alpha\ge 79\ge 40$).

During these  $\alpha/40$ steps, $c$ has only been able to move to a vertex $v''$ at distance at most $\alpha s/40$ from $v$, saying following some path $Q$. See Figure \ref{fig:pos} for an illustration.

\begin{figure}[htb] 
  \centering 
  \includegraphics[scale=1.2]{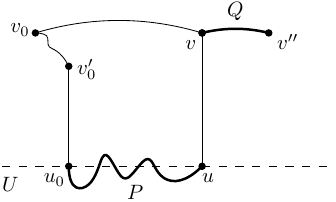}
  \caption{The paths $P$ and $Q$ in the proof of Claim \ref{C:move}.}
  \label{fig:pos}
\end{figure}

Observe that for any vertex $q\in Q$,
\begin{eqnarray*}
  d_H(q,U)&\ge& 
  d_H(v_0', U)-d_H(v_0',q)\\
  &\ge&
  d_H(u_0,v_0')-d_H(v_0',v_0)-d_H(v_0,v)-d_H(v,q)\\
  &\ge& \alpha s-\alpha s/4-s-\alpha s/40\\
  &\ge& (\tfrac{29}{40} \alpha-1)s\ge 55s.
\end{eqnarray*}
It follows that all vertices along the path $Q$ between $v$ and $v''$ are at distance more than $55s>32s$ from $U$, as in the definition of safe state 2. 
Moreover, as $|Q|\le \alpha s/40$ and $d_H(v'',U)\ge (\tfrac{29}{40} \alpha-1)s$, we indeed have $|Q|\le \tfrac14 d_H(v'',u)$, as desired. This shows that we can reach back safe state 2 in a finite number of moves. It only remains to check that rule \cref{r2} has been respected during these moves (we have already considered rule \cref{r1} above).

To prove that \cref{r2} also holds while $c$ moves along $Q$ (and $c_U$ moves along $P$), consider that $c$ lies at some vertex $q\in Q$ and $c_U$ lies in $p\in P$, and take any $r\in \pi(q)$. Observe first that
\begin{equation}\label{eq:2}
  d_H(q,v_0')\le d_H(q,v)+d_H(v,v_0)+d_H(v_0,v_0')\le \alpha s/40+s+\alpha s/4\le (11\alpha/40+1)s.
\end{equation}
We proved above that  $d_H(q,U)\ge (\tfrac{29}{40} \alpha-1)s\ge \alpha s/2$. Since $r\in U$, $d_H(q,r)\ge \alpha s/2$ and thus $\alpha s \le 2  d_H(q,r)$. We now obtain
\begin{eqnarray*}
  d_H(p,r) & \le & d_H(p,v_0')+d_H(v_0',q)+d_H(q,r)\\
           & \le & 2C^3(5\alpha/2+3)s+ \alpha s+(11\alpha/40+1)s+d_H(q,r)\\
           & \le & 7C^3\alpha s + d_H(q,r)\\
  & \le & (14C^3+1) d_H(q,r)\le 16C^3 d_H(q,r),
\end{eqnarray*}
where we have used $\alpha\ge 79\ge \tfrac{40\cdot 7}{29}$ and thus $(6+\tfrac{11}{40})\alpha +7\le 7\alpha$ in the third inequality.
This shows that \cref{r2} holds for any intermediate step, and concludes the first part of the claim.

\medskip

For the second part of the claim, assume that the cop $c$ is located on $v$, $c_U$ is located on $u\in U$, and the robber is located on $x\in U^+$. Assume that the robber is reachable from $c$, that is $d_H(x,v)\le \rho$.  Let $y\in \pi(x)$ and $w\in \pi(v)$. Since $x\in U^+$, $d_H(x,y)\le C$. By definition of $w$, $d_H(v,w)\le d_H(v,y)\le \rho+C$. On the other hand, it follows from   \cref{r2} that $d_H(u,w)\le 16C^3 d_H(v,w)\le 16C^3(\rho+C)$, and thus \[d_H(x,u)\le d_H(x,v)+d_H(v,w)+d_H(w,u)\le \rho+\rho+C+16C^3(\rho+C)=\rho_U.\] It follows that the robber is reachable from $c_U$ (with reach $\rho_U$), as desired. This concludes the proof of Claim \ref{C:move}.\end{proofofclaim}

Now that the robber has mapped every cop $c$ in $H$ to some virtual cop $c_U$ moving according to rules \cref{r1} and \cref{r2} in $H$, it remains to map $c$ to some virtual cop $c_G$ in $G$. As each cop $c_U$ finishes its moves in $U=h(V(G))$, it suffices to consider the location $u\in U$ of $c_U$ and map $c_G$ to any vertex $x\in h^{-1}(u)$. Assume that $c_U$ moves from $u$ to $v$ in a single round (we thus have $d_H(u,v)\le s_U$). Then $c_G$ moves from $x\in h^{-1}(u)$ to some $y\in h^{-1}(v)$. As $d_G(x,y)\le C s_U+C^2$, $c_G$ can indeed move from $x$ to $y$ in a single step if we set its speed to $s_G=C s_U+C^2$. We also set the reach $\rho_G$ of the cops $c_G$ in $G$ to $\rho_G=C\rho_U+C^2$. Note that if a vertex $x$ lies at distance more than $\rho_G$ from a cop $c_G$ in $G$, then $h(x)$ lies at distance more than $\rho_U$ from $c_U$ in $H$.

\medskip

It remains to describe the strategy of the robber.
Assume first we are interested in the strong version of the game. So the cops announce their speed $s$. The robber then computes $s_U$ and $s_G$ as above (both parameters only depend on $s$) and chooses its own speed $s_r$ in $G$ according to his winning strategy in $G$ against $t$ cops running at speed $s_G$. It then announces his speed in $H$ as $s^*=5C^3 s_r$. The cops then announce their reach $\rho$ in $H$. The robber computes $\rho_U$ and $\rho_G$ as above, and chooses a ball $B_G$ of bounded radius in $G$ in which it can come back infinitely often while travelling at speed $s_r$ and remaining at distance more than $\rho_G$ from the cops $c_G$ in $G$ (who travel at speed $s_G$). Whenever the robber is (virtually) located on some vertex $x$ in $G$, it is (really) located at the same time on $h(x)$ in $H$. So when the robber moves from $x$ to $y$ in $G$ we have $d_G(x,y)\le s_r$ and thus $d_H(h(x),h(y))\le Cs_r+C$. By Claim \ref{C:moveShadow}, there is a path of length $2C^2(C s_r+C)+C^3\le 5C^3 s_r$ in $U^+$ between $h(x)$ and $h(y)$. Since the  speed of the robber in $H$ is $s^*=5C^3 s_r$ the robber can move from $h(x)$ to $h(y)$ in a single step, while staying in $U^+$.

\medskip

Since the robber in $G$ remains at distance more than $\rho_G$ from the cops $c_G$, the robber in $H$ remains at distance more than $\rho_U$ from the cops $c_U$ in $H$, and thus at distance more than $\rho$ from the cops $c$ in $H$, by Claim \ref{C:move}.

\medskip

Note that as the robber in $G$ comes back infinitely often to the ball $B_G$ in $G$, the robber in $H$ comes back infinitely often in $h(B_G)$, which is contained in a ball of bounded radius in $H$. So the robber indeed has a winning strategy for the strong game in $H$ against $t$ cops. It follows that $\sco(G)\le \sco(H)$, as desired. The proof for the weak version is the same, with the slight difference that $s$ and $\rho$ are announced by the cops before the robber needs to announce its own speed. 
\end{proof}

\section{Strong cop numbers and hyperbolicity}\label{sec:hy}

\subsection{Hyperbolicity}

There are multiple equivalent definitions of hyperbolicity in graphs, and more generally in geodesic metric spaces. Informally, a graph is hyperbolic when its shortest path metric is close to a tree metric.

The following definition of hyperbolicity based on the notion of slim geodesic triangles is one of the most commonly used in the literature. 
For every three vertices $x,y,z\in V(G)$, a \emph{geodesic triangle} $\triangle(x,y,z)$ is a subgraph of $G$ obtained by taking the union of three geodesics $P_1, P_2, P_3$ going respectively from $x$ to $y$, from $y$ to $z$ and from $x$ to $z$. Given such a geodesic triangle $\triangle(x,y,z)$, we say that it is \emph{$\delta$-slim} for for some real $\delta\geq0$ if for every $i\in \{1,2,3\}$, every vertex $u\in V(P_i)$ is at distance at most $\delta$ in $G$ from $\bigcup_{j\ne i} P_j$.
We say that a (connected) graph $G$ is \emph{$\delta$-hyperbolic} if all its geodesic triangles are $\delta$-slim. We say that $G$ is \emph{hyperbolic} if there exists some $\delta\geq0$ such that it is $\delta$-hyperbolic. It is not hard to observe that hyperbolicity is preserved under taking quasi-isometries. This allows to extend the definition of hyperbolicy to finitely generated groups: a finitely generated group is \emph{hyperbolic} if one (or equivalently all) of its Cayley graphs is hyperbolic.

\smallskip

In his seminal work, Gromov \cite{Gromov_hyperbolic} gave a number of characterisations of hyperbolic graphs. Among them, he proved that a group is hyperbolic if and only if it is finitely presented and admits a linear isoperimetric inequality. Many proofs in \cite{Gromov_hyperbolic} extend to general geodesic spaces, and 
from a graph perspective, a sufficient condition implying linear isoperimetric inequality, and thus also hyperbolicity is the following, proved in \cite{Bowditch_hyperbolic}.

\begin{theorem}[Lemma 6.1.2 in \cite{Bowditch_hyperbolic}]
\label{thm: linear-isoperimetric}
Let $G$ be a graph and assume that there exists $D, K\geq 0$ such that for every cycle $C$ of $G$ of length at least $K$, there exists some subpath $P$ of $C$ of length at most $D$ with endpoints $x,y\in V(C)$, such that $d_G(x,y)<d_P(x,y)$. Then $G$ is hyperbolic.
\end{theorem}

\subsection{Hyperbolicity and strong cop number}

In \cite[Theorem 3.2]{lee2023coarse}, the authors proved that every hyperbolic graph has strong cop number $1$. We prove that this is indeed an equivalence.

\begin{theorem}
 \label{thm: hyperbolic}
 For any connected graph $G$,  $\sco(G)=1$ if and only if $G$ is hyperbolic.
\end{theorem}


\begin{proof}
By \cite[Theorem 3.2]{lee2023coarse} we only need to show that if $\sco(G)=1$ then $G$ is hyperbolic. 
Assume for the sake of contradiction that $G$ is not hyperbolic, then by Theorem \ref{thm: linear-isoperimetric}, for every $t$, there is a cycle $C$ of length at least $2t$ in which all subpaths of length $t$ are geodesic in $G$.

We now describe a winning strategy for the robber against a single cop. This will show that $\sco(G)>1$, as desired. 

\begin{itemize}
    \item The cop chooses some speed $s_c=s\ge 1$.
    \item The robber chooses speed $s_r=100s$.
    \item The cop chooses its reach $\rho$.
    \item The robber chooses a cycle $C$ in $G$ of length at least $400\rho s$ such that all subpaths of length $200\rho s$ are geodesic, and then defines $R$ as the radius of $C$ (the goal of the robber will be to remain in $C$ forever).
\end{itemize}

We say that the robber is \emph{safe} if the distance $d_G(c,r)$ between the cop and the robber in $G$ is at least $10\rho s+s$. The strategy of the robber is the following. It starts on $C$, in a safe position (this is possible since otherwise $C$ would have radius at most $10\rho s+ s$, which would contradict the fact that it contains geodesic paths of length $200\rho s$). As long as the robber is safe, he does not move. If the robber is not safe anymore after the cop's move, it means that $d_G(c,r)< 10\rho s+s$. Note that since the robber was safe at the previous round and the cop only moves at speed $s$, we also have $d_G(c,r)\ge 10\rho s$. 

Let us denote by $i$ the current round, and let $c_i$ and $r_i$ denote the positions of the cop and the robber at round $i$. During $\rho$ consecutive rounds, the robber is now going to move along $C$ to reach  a vertex located at distance $100\rho s$ from its current location (this is possible, since the speed of the robber is $100s$, and the subpath of $C$ between it current location and its destination is a geodesic). In order to decide in which direction to go (left or right), the robber considers the distance $d_G(c_{i},C)$ between the cop and the cycle $C$ at round $i$.

If $d_G(c_{i},C)>\rho s+\rho$, then the robber can go either left or right (it does not matter). Indeed, in $\rho$ rounds the cop will only be able to reach a vertex at distance at most $\rho s$ from its  location $c_{i}$ at round $i$, so it will remain at distance greater than $\rho$ from $C$ during the $\rho$ rounds, unable to reach the robber (which remains on $C$ the whole time).
 
Assume now that $d_G(c_{i},C)\le \rho s+\rho$. Let $x^-,x^+$ be the two endpoints of the subpath $P$ of $C$ of length $200\rho s$ centered in the position $r_{i}$ of the robber at round $i$, and let $P^-$ and $P^+$ be the two subpaths of $P$ between $x^-$ and $r_{i}$, and between $r_{i}$ and $x^+$, respectively. Note that by definition of $C$, $P$ and its two subpaths are geodesics. We claim that $c_{i}$ cannot be at distance at most $\rho s+\rho$ from both $P^-$ and $P^+$: indeed if there were vertices $v^-\in P^-$ and $v^+\in P^+$ at distance at most  $\rho s+\rho$ from $c_{i}$, then \[10\rho s\le d_G(c_{i},r_{i})\le d_G(c_{i},v^-)+d_G(v^-,r_{i})\le d_G(v^-,r_{i})+\rho s+\rho, \] and thus $d_G(v^-,r_{i})\ge 9\rho s-\rho$, and similarly $d_G(v^+,r_{i})\ge 9\rho s-\rho$. Since $P^-\cup P^+$ is a geodesic, and $v^-,r,v^+\in C$, $d_G(v^-,v^+)=d_G(v^-,r_{i})+d_G(v^+,r_{i})\ge 18\rho s-2\rho$. This contradicts the fact that $d_G(v^-,v^+)=d_G(v^-,c_{i})+d_G(v^+,c_{i})\le 2\rho s +2\rho$, since $s \ge 1$. This shows that the cop remains at distance more than $\rho$ from 
one of $P^+$ and $P^-$, say $P^+$ by symmetry, during the $\rho$ rounds following round $i$. In particular the robber can travel along $P^+$ to reach $x^+$ (at distance $100\rho s$ from the location $r_{i}$ of the robber at round $i$), without being captured by the cop. 

In both cases, at round $i+\rho$, the robber is on $C$, at distance $100\rho s$ from its location $r_{i}$ at round $i$. The cop and the robber were at distance at most $10\rho s+s$ at round $i$, and the cop has only been able to travel a distance at most $\rho s$ between rounds $i$ and $i+\rho$. It follows that at round $i+\rho$, the cop and robber are at distance at least $100\rho s-(10\rho s+s)-\rho s=89\rho s-s$. Since $\rho\ge 1$, this distance is at least $10\rho s+s$ and thus the robber is safe at round $i+\rho$. The robber can then wait until it is not safe anymore, and repeat forever the same procedure as above.
\end{proof}

\section{Weak cop numbers and asymptotic minors}\label{sec:am}

In Subsection \ref{S:minors}, we prove that every infinite connected graph which admits some finite graph of treewidth at least $k$ as an asymptotic minor has weak cop number at least $k+1$. We observe in Subsection \ref{S:wcn1} that it implies that an infinite connected graph has weak-cop number $1$ if and only if it is quasi-isometric to a tree. In Subsection \ref{S:virtually-free} we deduce that a locally finite quasi-transitive graph has bounded treewidth if and only if it has weak cop number $1$. In particular, the finitely generated groups with weak cop number $1$ are exactly the virtually free groups, giving a positive answer to \cite[Question 1.4]{Cornect_Martinez24}.

\subsection{Asymptotic minors and weak cop number}
\label{S:minors}

For every $D\geq 1$ and every graphs $H,G$, a \emph{$D$-fat $H$-minor} of $G$ consists in a collection $(C_x)_{x\in V(H)}$ of non-empty connected subgraphs of $G$, together with a collection $(P_e)_{e\in E(H)}$ of paths of $G$ such that
\begin{itemize}
 \item for every distinct vertices $x, y\in V(H)$, $d_G(C_x, C_y)\geq D$;
 \item for every edge $xy\in E(H)$, $P_{xy}$ is a path connecting a vertex from $C_x$ to a vertex of $C_y$, and which is internally disjoint from $\bigcup_{z\in V(H)}C_z$;
 \item for every $x\in V(H)$ and every $yz\in E(H)$ such that $x\notin \{y,z\}$, $d_G(C_x, P_{yz})\geq D$;
 \item for every distinct edges $xy, x'y'\in E(H)$, $d_G(P_{xy}, P_{x'y'})\geq D$.
\end{itemize}

If $G$ admits a $D$-fat $H$-minor for every $D\geq 1$, then we say that $H$ is an \emph{asymptotic minor} of $G$, and we write $H\preccurlyeq_{\infty} G$.
As noticed in \cite{GP23}, the property of having a fixed graph $H$ as an asymptotic minor is invariant under  quasi-isometry.

\medskip

In a graph $G$, we say that two subsets $X,Y\subseteq V(G)$ \emph{touch} if they intersect or if there is an edge $xy\in E(G)$ with $x\in X$ and $y\in Y$.
A \emph{haven} of order $k$ in a graph $G$ is a function $\beta$ which maps every subset $X\subseteq V(G)$ of less than $k$ vertices to a non-empty connected component of $G-X$, such that for every two sets $X,Y\subseteq V(G)$ of less than $k$ vertices, $\beta(X)$ and $\beta(Y)$ touch. An important property of havens is their monotonicity: if $X\subseteq Y$ are two sets of less than $k$ vertices, then $\beta(Y)\subseteq \beta(X)$ (see \cite{AST90} for more details on monotonicity in havens).

\smallskip

For a finite graph $G$, we denote by $\mathrm{bn}(G)$ the maximum order of a haven in $G$. The following result of Seymour and Thomas \cite{ST93} shows how $\mathrm{bn}(G)$ is related to the treewidth $\tw(G)$ for any finite graph $G$.

\begin{theorem}[\cite{ST93}]
\label{thm:bramble}
 For every finite graph $G$, we have
 $\mathrm{bn}(G)=\tw(G)+1$.
\end{theorem}

We say that a subset $X$ of vertices of a graph $G$ is \emph{connected} if the subgraph $G[X]$ of $G$ induced by $X$ is connected. Note that a haven $\beta$ of order $k$ in a graph $G$ maps subsets $X\subseteq V(G)$ of order less than $k$ to non-empty connected  subsets $\beta(X)$ of $V(G)$.

\medskip

Our main result in this section is the following.

\begin{theorem}
 \label{thm: minor}
 Let $H$ be a finite graph, and let $G$ be any connected graph such that $H\preccurlyeq_{\infty}G$. Then 
 $$\wco(G)\geq \tw(H).$$
\end{theorem}

\begin{proof}
We set $t:= \tw(H)$, and show that the robber always has a winning strategy in the weak variant of the Cops and robber game against $t-1$ cops. Up to considering a connected component of $H$ with the same treewidth, we can assume in the remainder of the proof that $H$ is connected. 
By Theorem \ref{thm:bramble}, there exists a heaven $\beta$ of order $t+1$ in $H$.

\smallskip
 
 Assume that the $t-1$ cops initially choose speed $s_c$ and reach $\rho$. We set $D:=2\cdot (s_c+\rho+1)$, and let $((C_u)_{u\in V(H)}, (P_e)_{e\in E(H)})$ be a $D$-fat $H$-minor in $G$.
 The main goal of the robber will be to map the position of each cop in $G$ to some virtual position in $H$ (depending to their respective distances to the different elements of the fat $H$-minor in $G$), apply a natural strategy to (virtually) evade the cops in $H$ using the haven $\beta$, and then map this back to a real position in $G$. This motivates the following definitions.

 \medskip

 For any subset $U\subseteq V(H)$, we let
 $$U^G:=\left(\bigcup_{u\in U}C_u \right)\cup \left(\bigcup_{e\in H[U]}V(P_{e})\right),$$
 and we write $C:=V(H)^G$ (that is, $C$ consists of all the vertices of the different elements $C_x$ and $P_e$ of the  $D$-fat $H$-minor in $G$).
 As $H$ is finite, note that we may assume that every $C_u$ is finite, hence all subsets $U^G\subseteq V(G)$ are also finite. Observe also that by definition of a fat minor, for any non-empty connected subset $U\subseteq V(H)$, $U^G\subseteq C$ is a non-empty connected subset of $V(G)$.
 
 \smallskip
 
 The robber defines his speed $s_r$ as the length of a longest path of $G[C]$ (note that since we assumed that $H$ is connected, it follows from the definition of a fat minor that $C$ induces a connected subgraph of $G$), and 
 chooses any ball of finite radius in $G$ that contains $C$. We will show that the robber has a winning strategy in which he remains in $C$ at each step. Note that by definition of $s_r$, the robber can travel in a single round along any path from $G[C]$ that lies at distance more than $\rho$ from all the cops.
 
 \medskip

 We now explain how the robber can map vertices of $G$ to some virtual position in $H$.
 For every vertex $x\in V(G)$, we define  $x_H\subseteq V(H)$ as follows. We start with $x_H=\emptyset$ for any $x\in V(G)$, and 
 \begin{itemize}
  \item for any $u\in V(H)$  such that 
  $d_G(x,C_u)\leq \frac{D}{2} -1$, we add $u$ to the set $x_H$;
  \item if $x$ is at distance  more than $\frac{D}{2}-1$ from any set $C_u$ in $G$, for any  $e\in V(H)$ such that $d_G(x,P_e)\leq \frac{D}{2} -1$, we add  one of the two endpoints of $e$ in $H$ to $x_H$.
  \end{itemize}

\begin{claim}
 \label{clm: unique_branch}
 For every $x\in V(G)$, $|x_H|\leq 1$.
\end{claim}

\begin{proofofclaim}
  This immediately follows from the fact that for every two distinct vertices $u,v\in V(H)$, and every two distinct edges $e,f\in E(H)$, we have $d_G(C_u, C_v)\geq D$, $d_G(P_e, P_{f})\geq D$ and $d_G(C_u, P_e)\geq D$ if $e$ is not incident to $u$.
\end{proofofclaim}

For every subset $X\subseteq V(G)$, we set $X_H:=\bigcup_{x\in X}x_H$. Note that by Claim \ref{clm: unique_branch}, $|X_H|\le |X|$ for every subset $X\subseteq V(G)$.  

\begin{claim}
  \label{clm: dist}
  Consider two subsets $X\subseteq V(G)$ and $U\subseteq V(H)$ such that $U\cap X_H=\emptyset$. Then $d_G(X,U^G)\ge D/2=s_c+\rho+1$, and in particular for any $u\in U$ we have $d_G(X,C_u)>s_c+\rho$.
 \end{claim}
 
 \begin{proof}
  Assume for the sake of contradiction that some vertex $y\in U^G$ lies at distance at most $D/2-1$ from  some $x\in X$. By definition of $U^G$, $y$ must be in some $C_u$ with $u\in U$, or in some $P_{uv}$ with $u,v\in U$, so there exists a vertex $u\in U$ such that $x_H=\{u\}$, and in particular $u\in X_H$. It follows that $u\in U\cap X_H\ne \emptyset$, a contradiction.
 \end{proof}

For every subset $X\subseteq V(G)$ of size at most $t$, we set $X^\perp:=\beta(X_H)^G$. The next claim shows that $X^\perp$ is a connected subset of vertices of $G$ that is sufficiently far from $X$.
 
\begin{claim}
  \label{clm: bramble}
  For every subset $X\subseteq V(G)$ of size at most $t$, $X^\perp\subseteq C$ is a non-empty connected subset of $V(G)$ which lies at distance at least $D/2=s_c+\rho+1$ from $X$ in $G$.
 \end{claim}

 \begin{proofofclaim}
 Note that as $|X|\le t$, we have $|X_H|\le t$ and thus $U:=\beta(X_H)$ is a non-empty connected component of $H-X_H$, so $U\cap X_H=\emptyset$. As $U$ is non-empty and connected in $H$, $X^\perp=U^G$ is non-empty and connected in $G$. The property that $d_G(X,U^G)\ge D/2$ follows directly from Claim \ref{clm: dist}.
 \end{proofofclaim}

Assume that the starting positions of the cops are given by a set $Y_0\subseteq V(G)$ of size at most $t-1$. The robber then starts at an arbitrary vertex $r_0\in Y_0^\perp$. Note that  by Claim \ref{clm: bramble}, $Y_0^\perp\subseteq C$ is a non-empty connected subset of $V(G)$ lying at distance at least $s_r+\rho+1$ from $Y_0$. At the end of each round $i\ge 1$ we denote the set of vertices occupied by the cops by $Y_i$ and the location of the robber by $r_i$. Observe that $|Y_i|\le t-1$ for every $i\ge 0$.

\medskip

We now prove the following by induction on $i\ge 0$:

\begin{itemize}
\item[$(*)$] \emph{the robber can stay in $G[C]$ forever without being caught by the cops, and moreover for every $i\ge 0$, $r_i\in Y_i^\perp\subseteq C$.}
\end{itemize}

As $C$ is contained within a ball of bounded radius, this readily implies that the robber has a winning strategy against $t-1$ cops in the weak game, and thus $\wco(G)\geq t=\tw(H)$, as desired. So it only remains to prove $(*)$.

\medskip

By the definition of $r_0$, $(*)$ holds for $i=0$. We can thus assume that $i\ge 1$. We now write $A=Y_{i-1}$ and $B=Y_i$ to simplify the notation.

\medskip

Consider the situation of the robber at round $i\ge 1$, just after the cops have moved from $A=Y_{i-1}$ to their current position in $B=Y_i$. The robber occupies the vertex $r_{i-1}\in A^\perp$ and wants to move  to $r_i\in B^\perp$. We say that the robber is \emph{safe} if it stays at distance more than $\rho$ from $B$ (so out of reach from the cops).

\begin{claim}\label{cl:yi-1}
We have $d_G(B,A^\perp)>\rho$, and in particular the robber can travel safely in $G[C]$ to any vertex of $A^\perp$.
\end{claim}

\begin{proof}
Recall that by Claim \ref{clm: bramble}, $A^\perp\subseteq C$ is a non-empty connected subset of $V(G)$ which lies at distance at least $s_c+\rho+1$ from $A$ in $G$. Since the cops travel at speed $s_c$, $d_G(A^\perp,B)\ge \rho+1$ and thus the robber can travel safely within $A^\perp\subseteq C$ to reach any vertex of $A^\perp$.
\end{proof}

Recall that $A^\perp=\beta(A_H)^G$ and $B^\perp=\beta(B_H)^G$. By definition of a haven, $\beta(A_H)$ and $\beta(B_H)$ touch in $H$. If they have a vertex in common (in $H$), then $A^\perp$ and $B^\perp$ also have a vertex in common (in $G$), and the property that $(*)$ holds at the end of round $i$ follows from Claim \ref{cl:yi-1}. So we can assume that
$\beta(A_H)$ and $\beta(B_H)$ are disjoint, and thus by definition of a haven there is an edge connecting $\beta(A_H)$ to $\beta(B_H)$ in $H$. 

\smallskip

The goal of the robber will now be to escape from $A^\perp$. What would prevent him to do so is if all the escape routes from $A^\perp$ were blocked by some cop. So let us say that an edge $uv\in E(H)$ and the corresponding path $P_{uv}$ in $G$ are  \emph{guarded} if 
\begin{itemize}
    \item $u\in \beta(A_H)$;
    \item $v\notin \beta(A_H)$;
    \item there exists $y\in B$ such that $d_G(y,P_{uv})\le \rho$ and $u\in y_H$ (and thus, by Claim \ref{clm: unique_branch}, $y_H=\sg{u}$).
\end{itemize}

Let $e_1, \ldots,e_s$ denote the edges of $H$ that are guarded.  For every $1\le j \le s$, we write $e_j=u_jv_j$, where $u_j\in \beta(A_H)$ and $v_j\notin \beta(A_H)$. Let $Z\subseteq B$ be the set of vertices occupied by the cops that do not guard any path $P_{e_j}$. As there are $t-1$ cops, and each cop can guard at most one path, we have $s+|Z|\le t-1$.

\begin{claim}
  \label{clm: bad-free}
  For every vertex $v\in V(H)$ and every $y\in B-Z$, we have $d_G(y, C_v)>\rho$. 
 \end{claim}
 
 \begin{proofofclaim}
  Assume for a contradiction that $d(y, C_v)\leq \rho$ for some $v\in V(H)$ and $y\in B-Z$, and let $e\in E(H)$ be the edge guarded by some cop occupying $y$. Note that we then have $y_H=\sg{v}$.
  As $P_e$ is at distance at least  $D>2\rho$ from every set $C_{w}$ such that $w$ is not an endpoint of $e$, $v$ must be an endpoint of $e$. By Claim \ref{cl:yi-1}, $d_G(B,A^\perp)>\rho$ and thus $v\notin \beta(A_H)$. By the definition of a guarded edge, this implies that $y_H\ne \{v\}$, a contradiction. 
  \end{proofofclaim}

For any $1\le j \le s$, we define \[U_j:=\{u_1,\ldots, u_j\}\cup \{v_j,\ldots, v_s\}.\] Note that each set $U_j$ contains at most $s+1$ vertices and thus $|U_j\cup Z_H|\le t$ for any $1\le j \le s$. We now show that for any $j=1,\ldots,s$, the robber can travel safely to \[F_j:=\beta(U_j\cup Z_H)^G\subseteq C.\] 
As $U_{s}\cup Z_H \supseteq \{u_j:1\le j \le s\}\cup Z_H=B_H$, we have  $\beta(U_{s}\cup Z_H)\subseteq \beta(B_H)$ by the monotonicity property of havens \cite{AST90}. It follows that $F_{s}=\beta(U_{s}\cup Z_H)^G\subseteq \beta(B_H)^G=B^\perp$, and thus if the robber can travel safely to $F_{s}$ it can also travel safely to $B^\perp$, as desired. So in order to prove $(*)$, and thus conclude the proof of the theorem, it only remains to show that for  $j=1,\ldots,s$, the robber can travel safely to $F_j$ (while remaining in $G[C]$ the whole time).

\medskip

We start with the base case $j=1$.

  \begin{claim}
   \label{clm: C1}
   The robber can travel safely  in $G[C]$ to $F_1$.
  \end{claim}
  
  \begin{proofofclaim}
    By definition of a haven, $\beta(A_H)$ and $\beta(U_1\cup Z_H)$  touch in $H$. If they intersect, then $A^\perp=\beta(A_H)^G$ and $F_1=\beta(U_1\cup Z_H)^G$  intersect in $G$ and the robber can travel safely in $G[C]$ to  $F_1$ by Claim \ref{clm: bramble}. Assume now that $\beta(A_H)$ and $\beta(U_1\cup Z_H)$  are disjoint, and thus there is an edge $e=uv\in E(H)$ with $u\in \beta(A_H)$ and $v\in \beta(U_1\cup Z_H)$.
    
    We now show that $d_G(B,P_e)>\rho$, so the robber can use the path $P_e$ to travel safely from $A^\perp$ to some vertex of $F_1$. Assume for the sake of contradiction that there is some vertex $y\in B$ such that $d_G(y,P_e)\le \rho$.

    Assume first that $y\in Z$. By definition of a haven, $\beta(U_1\cup Z_H)$ is a component of $H-(U_1\cup Z_H)$, and thus   $v\notin y_H$. As $d_G(y, P_e)\leq \rho$, we then have $y_H=\sg{u}$.
    In particular, as $v\notin \beta(A_H)$, it follows from the definition of a guarded edge that  $e$ is guarded by $y$, which contradicts the fact that $y\in Z$. 
    

    So we can assume that $y\in B-Z$, and thus $d_G(y,C_v)>\rho$ by Claim \ref{clm: bad-free}. It follows that $y$ is guarding the edge $uv$, and thus there is an integer $1\le j \le s$ such that $u=u_j$ and $v=v_j$. As $\{v_j:1\le j \le s\}\subseteq U_1$, this contradicts the fact that $v=v_j\in \beta(U_1\cup Z_H)$.
  \end{proofofclaim}

  \begin{claim}
   \label{clm: Ci}
   For every $j\in [s]$, if the robber can travel safely  in $G[C]$ to some vertex of $F_j$, then he can also travel safely  in $G[C]$ to any vertex of $F_j$.
  \end{claim}
  
  \begin{proofofclaim}
      We first show that for every $y\in B$ and every $u\in \beta(U_j\cup Z_H)$, $d_G(y,C_u)>\rho$. This follows from Claim \ref{clm: bad-free} if $y\in B-Z$, so it remains to consider the case $y\in Z$.  As $\beta(U_j\cup Z_H)$ is disjoint from $Z_H$, it directly follows from Claim \ref{clm: dist} that $d_G(Z,C_u)>\rho$ for any $u\in \beta(U_j\cup Z_H)$, and thus $d_G(y,C_u)>\rho$, as desired.
    
      To prove that the robber can travel safely to any vertex of $F_j$, it remains to show that for every $e=uv\in \beta(U_j\cup Z_H)$, and $y\in B$, we have $d_G(y,P_e)>\rho$. Assume for the sake of contradiction that some vertex $y\in B$ is such that $d_G(y,P_e)\leq \rho$, for some $e=uv\in \beta(U_j\cup Z_H)$. In particular, $y_H\subseteq \{u,v\}$.
Observe that since $U_j$ contains at least one of $u_{j'},v_{j'}$ for any $1\le j'\le s$, $\beta(U_j\cup Z_H)$ cannot contain both endpoints from an edge $e_{j'}$ with $1\le j'\le s$. This implies that $y\in Z$. As $\beta(U_j\cup Z_H)\cap Z_H=\emptyset$, it then follows from Claim \ref{clm: dist} that $d_G(y,P_e)>\rho$, a contradiction.
  \end{proofofclaim}

  \begin{claim}
   \label{clm: nextCi}
   For every $1\le j< s$,
   if the robber can travel safely in $G[C]$ to $F_j$, then he can also travel safely in $G[C]$ to some vertex of $F_{j+1}$.
  \end{claim}

\begin{proofofclaim}
  By the definition of a haven, $\beta(U_j\cup Z_H)$ and $\beta(U_{j+1}\cup Z_H)$ touch in $H$. If they intersect, then $F_j$ and $F_{j+1}$ intersect and the result follows from Claim \ref{clm: Ci}. So we can assume that $\beta(U_j\cup Z_H)$ and $\beta(U_{j+1}\cup Z_H)$ are disjoint, and thus there is an edge $e=uv\in E(H)$ with $u\in \beta(U_j\cup Z_H)$ and $v\in \beta(U_{j+1}\cup Z_H)$. 

  We claim that $d_G(B,P_e)>\rho$, so that the robber can safely use $P_e$ to travel from $C_j$ to $C_{j+1}$. If this is not the case, then $d_G(y,P_e)\le \rho$ for some $y\in B$. As $Z_H$ is disjoint from $\beta(U_j\cup Z_H)$ and $\beta(U_{j+1}\cup Z_H)$, it follows from Claim \ref{clm: dist} that $d_G(Z,P_e)>\rho$, so we can assume that  $y\in B-Z$. Then there is an integer $1\le {j'}\le s$ such that $e=e_{j'}$, and thus $\{u,v\}=\{u_{j'},v_{j'}\}$. If $j'\le j$ then since $u_{j'}\in U_j\cap U_{j+1}$, we have $u_{j'}\notin \beta(U_j\cup Z_H)$ and $u_{j'}\notin\beta(U_{j+1}\cup Z_H)$. If $j'\ge  j+1$ then since $v_{j'}\in U_j\cap U_{j+1}$, we have $v_{j'}\notin \beta(U_j\cup Z_H)$ and $v_{j'}\notin\beta(U_{j+1}\cup Z_H)$. This contradicts the fact that $\{u,v\}=\{u_{j'},v_{j'}\}$.
  \end{proofofclaim}

  We have thus proved that the robber can reach $B^\perp$ from $A^\perp$ while remaining in $G[C]$ and staying at distance more than $\rho$ from the cops (located in $Y_i$). As the speed of the robber is equal to the length of a longest path in $G[C]$, the robber can travel in a single step from $A^\perp$ to $B^\perp$, which implies that $(*)$ holds and  concludes the proof of Theorem \ref{thm:bramble}.
 \end{proof}

Note that the lower bound from Theorem \ref{thm: minor} is optimal, in the sense that there exist graphs $G$ admitting some finite graph $H$ as an asymptotic minor, and such that $\wco(G)=\tw(H)$. To see this, we let $t\geq 1$ and $H:=K_t$, and we denote for every $i\geq 1$ by $H^{(i)}$ the graph obtained from $H$ by subdividing each edge $i$ times. We now let $G$ be obtained from the disjoint union of the graphs $H^{(i)}$, for $i\ge 1$, and an infinite one-way path $P=u_1,u_2,\ldots$ by choosing one vertex $v_i$ in each $H^{(i)}$, and identifying each vertex $v_i$ with $u_i$. Clearly, we have $H\preccurlyeq_{\infty} G$. Moreover, it is not hard to check that in this example, $\wco(G)\leq\tw(H)=t-1$. 
\medskip

We suspect that the relation between the weak cop number of a graph $G$ and the treewidth of the asymptotic minors of $G$ is deeper than suggested by Theorem \ref{thm: minor}.

\begin{question}
 \label{q: minors}
 Is it true that if a locally finite connected graph $G$ excludes some finite planar graph $H$ 
 as an asymptotic minor, then
 it has finite weak cop number?   If so, does there exist some 
 function $f:\mathbb N\to \mathbb N$ such that \[\wco(G)\leq f\big(\max\sg{\tw(H): |H|<\infty~\text{and}~H\preccurlyeq_{\infty} G}\big)?\] Can we choose $f=\mathrm{id}_{\mathbb N}$? 
\end{question}

We note that the assumption that the graph $G$ is locally finite is necessary in Question \ref{q: minors}, as the graph described in Example \ref{ex: tw2-wcop-infty} below excludes $K_4$ as an asymptotic minor (and thus also the $3\times 3$ grid), while its weak cop number is infinite.  

\medskip

We will discuss Question \ref{q: minors} again in the next section, in connection with a conjecture of Georgakopoulos and Papasoglu (Conjecture \ref{conj:GP}).
We also wonder if Question \ref{q: minors} admits a positive answer for general infinite graphs (not necessary locally finite), when replacing $\wco$ by its variant $\wco'$ that we will define in Section \ref{sec: alt}.

\subsection{Graphs quasi-isometric to graphs of bounded treewidth}
\label{S:wcn1}

It is not hard to check that trees have weak cop number $1$. In particular, as the weak cop number is a quasi-isometry invariant, every graph which is quasi-isometric to a tree has weak cop number $1$. We now deduce from Theorem \ref{thm: minor} that the converse implication also holds.

\begin{corollary} 
 \label{cor: wcn1}
 Let $G$ be a connected graph. Then $G$ has weak cop number $1$ if and only if $G$ is quasi-isometric to a tree. 
\end{corollary}

\begin{proof}
By the remark above, it suffices to prove that every graph with weak cop number $1$ is quasi-isometric to a tree. We show the contrapositive. Let $G$ be a graph which is not quasi-isometric to a tree. Georgakopoulos and Papasoglou \cite[Theorem 3.1]{GP23} recently proved that a graph is quasi-isometric to a tree if and only if it excludes $K_3$ as an asymptotic minor, hence we have $K_3\preccurlyeq_{\infty} G$. In particular, Theorem \ref{thm: minor} implies that $\wco(G)\geq \tw(K_3)=2$, as desired.
\end{proof}

\begin{proposition}
 \label{prop: tw-ub}
 If $G$ is a locally finite connected graph which is quasi-isometric to a graph $H$, then $\wco(G)\leq \tw(H)+1$. 
\end{proposition}

\begin{proof}
 Since the result is trivial when $H$ has unbounded treewidth, we can assume that $H$ has finite treewidth. As the weak cop number is invariant under quasi-isometry, it is enough to show that every locally finite graph $G$ of treewidth at most $k\in \mathbb{N}$ has weak cop number at most $k+1$.
 For this, we consider a tree-decomposition $(T, \mathcal V)$ of $G$ of width at most $k$, where $\mathcal V=(V_t)_{t\in V(T)}$. We may assume without loss of generality that no two adjacent bags are equal. 
 We describe a winning strategy for $k+1$ cops to win the weak version of the cops and robber game. We let $s_c:=1$ and $\rho:=1$ be respectively the speed and reach of the cops, and let $s_r, B$ denote respectively the speed and the ball of finite radius in $G$ that the robber chooses. We fix an arbitrary node $t_0\in V(T)$. At the initial step, the cops occupy all vertices from $V_{t_0}$.

 We now prove by induction on $i\geq 0$ that there exists integers $\Delta_0, \ldots, \Delta_i$ and nodes $t_0,\ldots, t_i\in V(T)$ such that if we set for each $j\in \sg{0,\ldots,i}$, $\tau_j:=\Delta_0+\ldots+\Delta_j$, then the cops have a strategy in which: 
 \begin{itemize}
  \item for each $j\in \sg{0,\ldots, i}$, at the end of step $\tau_j$, the cops occupy all the vertices from the bag $V_{t_j}$;
  \item for each $j\in \sg{0,\ldots, i-1}$, if $x\in V(G)$ denotes the position of the robber at the end of step $\tau_j$, if $T_x$ is the subtree of $T$ induced by the nodes $t\in V(T)$ such that $x\in V_t$, and if $T'$ denotes the unique connected component of $T-t_j$ containing $T_x$, 
  then $t_{j+1}$ is the unique neighbour of $t_j$ in $T'$, and none of the vertices $t_0, \ldots, t_{j}$ lies in $T'$.
\end{itemize}
Note that the above properties are trivially satisfied for $i=0$. We now let $i\geq 1$ and assume that we already found $\Delta_0, \ldots, \Delta_{i-1}$ and $t_0, \ldots, t_{i-1}\in V(T_0)$ satisfying all the above conditions. We let $x\in V(G)$ denote the position of the robber at the end of step 
$\tau_{i-1}$, and, reusing the notations from the second item, we let $t_i$ denote the unique neighbour of $t_{i-1}$ in $T'$. We let $(X,Y)$ denote the edge-separation of $(T,\mathcal V)$ such that $V_{t_{i-1}}\subseteq X$ and $V_{t_i}\subseteq Y$. In particular, $x\in Y$. 

We claim that after a finite number $\Delta_i$ of steps, the cops can reach the configuration where they occupy all the vertices of $V_{t_i}$, while ensuring that the robber does not leave $Y$ during their moves. 
We let $k'=|V_{t_{i}}\setminus V_{t_{i-1}}|\leq k$. 
Note that at the end of step $\tau_{i-1}$, there are at least $k+1-k'$ cops occupying all the vertices in $V_{t_{i-1}}\cap V_{t_i}$ (recall that two cops may occupy the same position, if $|V_{t_{i-1}}|<k+1$). 
As $G$ is connected, we can thus move $k'$  cops located either in $V_{t_{i-1}}\setminus V_{t_i}$, or located on a vertex of $V_{t_{i-1}}\cap V_{t_i}$ occupied by another cop, such that after a finite number $\Delta_i$ of steps, the cops occupy all the vertices from $V_{t_i}$, and such that for every $m\in \sg{0,\ldots,\Delta_i}$, at the end of step $\tau_{i-1}+m$, all the vertices of $V_{t_{i-1}}\cap V_{t_i}$ are occupied by cops. This ensures that the robber cannot leave $Y$ during all these steps if he does not want to be caught by the cops. In particular, this implies that the two items above are still satisfied for the value $i$, and thus they are satisfied for any $i\ge 0$ by induction. 

\medskip

To conclude the proof, we now show that there exists some $i\geq 0$ such that if $(X_i,Y_i)$ denotes the edge-separation of $(T,\mathcal V)$ such that $V_{t_{i-1}}\subseteq X_i$ and $V_{t_i}\subseteq Y_i$, then $B\subseteq X_i$. Note that it will immediately conclude the proof, as after step $\tau_i$, the cops will always occupy a separator between any position of the robber and $B$, ensuring that the robber cannot enter $B$ anymore. 
As $G$ is locally finite, $B$ must be finite, hence it is enough to show that for every $x\in B$, there exists some $i\geq 1$ such that for every $i'\geq i$, $x\notin Y_{i'}\setminus X_{i'}$. We let $x\in B$. Observe first that for every two integers $i\leq i'$, we have 
$Y_{i'}\subseteq Y_i$ and $Y_{i'}\setminus X_{i'}\subseteq Y_i\setminus X_i$. In particular, it implies that if $x\notin Y_{i}\setminus X_{i}$ for some $i\geq 0$, then $x\notin Y_{i'}\setminus X_{i'}$ for each $i'\geq i$. 
We now let $t\in V(T)$ be such that $x\in V_t$, and such that the distance $d_T(t, t_0)$ is minimum. We set $d:=d_T(t,t_0)$. Note that for every $i\geq 1$, if $x \notin X_i$, then we must have $d_T(t,t_i)=d_T(t,t_{i-1})-1$. In particular, this implies that $x\in X_d$, and therefore that $x\in X_{i'}$ for any $i'\geq d$, concluding our proof.
\end{proof}

The following example shows that the statement of Proposition \ref{prop: tw-ub} does not hold anymore if $G$ has a vertex of infinite degree (that is, if $G$ is not locally finite). 

\begin{proposition}\label{ex: tw2-wcop-infty}
There is a graph $G$ such that $G$ has a single vertex of infinite degree, $\tw(G)=2$ and $\sco(G)=\wco(G)=\infty$.
\end{proposition}
 
 \begin{proof}
 For every $i\geq 2$, we let $T_i$ denote the infinite $i$-regular tree. We let $H_i$ denote the graph obtained by adding a universal vertex $z_i$ in $T_i$ (a vertex adjacent to all vertices of $T_i$), and we let $G_i$ be the graph obtained from $H_i$ by subdividing each edge $i-1$ times (that is, by replacing each edge by a path of length $i$). Finally, we let $G$ be the graph obtained by taking the disjoint union of all graphs $G_i$, $i\ge 2$, and identifying all the vertices $z_i$ into some vertex $z$ (of infinite degree), see Figure \ref{fig: ex}.
 
 \begin{figure}[htb] 
  \centering 
  \includegraphics[scale=1]{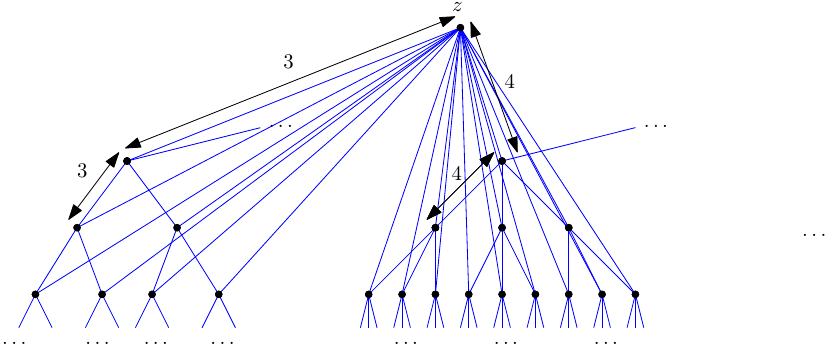}
  \caption{The graph described in Proposition \ref{ex: tw2-wcop-infty}. Blue lines represent subdivided edges.}
  \label{fig: ex}
\end{figure}

 Clearly, the graph $G-z$ is a tree, hence we have $\tw(G)=2$. Moreover, we claim that $\sco(G)=\infty$, and thus, as $\sco(G)\le \wco(G)$, $\wco(G)=\infty$. 

 To see this, we let $k\in \mathbb N$ and show that the robber has a winning strategy in the strong version of the cops and robber game in $G$ against $k$ cops. Assume that the $k$ cops choose speed $s_c$. We let $s_r:=4 s_c$ be the speed of the robber, and $\rho$ be the reach chosen by the cops afterwards. We set $C:=\max(k, 2(s_c+\rho))+1$, and consider the ball $B$ of radius $2C$ centered in $z$. For simplicity, and up to choosing a larger value for $C$, we will also assume that $C$ is a multiple of $4s_c$. Note that the whole subgraph $G_{C}$ of $G$ is contained in $B$. We now show that the robber has a winning strategy in which he will occupy at each step a vertex of $G_{C}$. To simplify the notation, we assume that $V(T_{C})\subseteq V(G_{C})\subseteq V(G)$, and for each edge $e\in E(H_{C})$, we let $P_e$ denote the corresponding path of length $C$ between two vertices of $V(T_C)$ in $G_{C}$ (and thus in $G$). 
 At every step of the game such that the cops occupy vertices $v_1, \ldots, v_k\in V(G)$, we say that a vertex $u\in V(T_C)$ is \emph{safe} if it is at distance more than $\frac{C}{2}$ from all cop positions $v_1, \ldots, v_k$ in $G$.
 We say that $u\in V(T_C)$ is \emph{super safe} if $u$ is at distance more than $\tfrac34 C$ from all cop positions $v_1, \ldots, v_k$ in $G$. 
 Note that our choice of $C$ guarantees that every safe vertex is out of  reach from the cops. 
 
 We now describe a winning strategy for the robber. After the cops chose their initial positions, the robber chooses as initial position any safe vertex in $V(T_C)$. Note that such a vertex always exists, as each vertex of $G$ lies at distance at most $C/2$ from a finite number of vertices of $V(T_C)$.
 As long as the vertex $u$ occupied by the robber is safe, the robber remains in his current location. We now assume that at some step $t\geq 1$, the cops move so that $u$ is not safe anymore, and we will show that, after a finite number of rounds, the robber can reach a node $w\in V(T_C)$ which will be safe again. In particular, this will immediately imply a winning strategy for the robber in the strong version of the cops and robber game. Note that as $u$ was safe just before the cops moved, at the beginning of step $t$, every cop position $v_i$ is at distance at least
 $\frac{C}{2}+1-s_c > \rho$ from $u$, so the cops did not catch the robber after their moves. As $u$ has degree at least $k+1$ in $T_C$, and $T_C$ is a tree, $u$ has a super safe neighbor $w$ in $T_C$ such that the path $P_{uw}$ in $G_C$ is not occupied by any cop.  It thus means that the robber can travel along the path $P_{uw}$ to reach $w$ in $\frac{C}{s_r}= \frac{C}{4 s_c}$ steps. Note that during these moves, each cop can only travel a distance of $s_c\cdot \frac{C}{4 s_c}=\frac{C}{4}$, and thus each cop is at distance  more than $\frac{3C}{4}-\frac{C}{4}\ge\frac{C}{2}$ from $w$ just after the robber has reached $w$. 
 It thus implies that $w$ is a safe vertex at the end of step $t+\frac{C}{s_r}$, as desired.
%
%
 \end{proof}

We obtain the following as a direct consequence of Proposition \ref{prop: tw-ub}.

\begin{corollary} 
 \label{cor: wcn2}
 Every connected graph with weak cop number at most $2$  is quasi-isometric to a graph of treewidth at most 2. Every connected locally finite graph which is quasi-isometric to a graph of treewidth at most 2 has weak cop number at most 3.
\end{corollary}

\begin{proof}
It was proved in \cite{AJKW24} that any graph which excludes $K_4$ as an asymptotic minor is  quasi-isometric to a graph of treewidth at most 2. This shows that any graph $G$ which is not quasi-isometric to a graph of treewidth at most 2 contains $K_4$ as an asymptotic minor, and thus $\wco(G)\ge \tw(K_4)=3$ by Theorem \ref{thm: minor}. This proves the first part of the statement. The second part is an immediate consequence of Proposition \ref{prop: tw-ub}.
\end{proof}

We now construct a locally finite graph of treewidth 2 and weak cop number 3, showing that the second part of the statement of Corollary \ref{cor: wcn2} is optimal. Let $H$ be the graph obtained from a triangle by first replacing every edge by two parallel edges, and then subdividing each edge exactly once. As before, for every $i\geq 1$, we denote by $H^{(i)}$ the graph obtained from $H$ by subdividing each edge $i$ times. We now let $G$ be obtained from the disjoint union of the graphs $H^{(i)}$, for $i\ge 1$, and an infinite one-way path $P=u_1,u_2,\ldots$ by choosing one vertex $v_i$ in each $H^{(i)}$, and identifying each vertex $v_i$ with $u_i$. Clearly, $\tw(G)=\tw(H)=2$, but it is not hard to check that $\wco(G)\ge 3$. 

\medskip

A natural question is whether some analogue of Corollaries \ref{cor: wcn1} and \ref{cor: wcn2} also holds for larger values of $k$.

\begin{question}
 \label{q: Qi-tw}
 Is it true that if a locally finite graph $G$ has finite weak cop number, then it is quasi-isometric to some graph of finite treewidth? If so, is it quasi-isometric to some graph of treewidth at most $f(\wco(G))$ for some $f:\mathbb N\to \mathbb N$? Can we choose $f(k)=k+1$? 
\end{question}

See \cite{NSS25,Hic25} for recent work on the structure of graphs that are quasi-isometric to graphs of bounded treewidth. 

\medskip

To summarize, every locally finite graph which is 
quasi-isometric to some graph of bounded finite treewidth has bounded weak cop number, and by Theorem \ref{thm: minor}, every graph $G$ with finite weak cop number 
excludes as an asymptotic minor every graph of treewidth $\wco(G)$. 
In \cite[Conjecture 9.2]{GP23}, Georgakopoulos and Papasoglou conjectured the following.

\begin{conjecture}[\cite{GP23}]\label{conj:GP}
If there is an integer $k$ such that a graph $G$ excludes the $k\times k$ grid as an asymptotic minor, then $G$ is quasi-isometric to a graph of bounded treewidth.
\end{conjecture}

Note that the $k\times k$ grid has treewidth $k$, and thus it follows from the paragraph above that
every graph with finite weak cop number excludes the $k\times k$ grid as an asymptotic minor for some integer $k$. As a consequence, the validity of Conjecture \ref{conj:GP} would directly imply a positive answer to the first part of Question \ref{q: Qi-tw}. Somewhat unfortunately, a counterexample to Conjecture \ref{conj:GP} was recently constructed in \cite{AD25}. It would be interesting to study the weak cop number of their construction: if it is finite, this would provide a negative answer to Question \ref{q: Qi-tw}, while if it is infinite, this would provide a negative answer to Question \ref{q: minors}.
Nevertheless, the discussion above still implies that if Conjecture \ref{conj:GP} is valid when $G$ is restricted to some class of graphs $\mathcal{G}$, then the first part of Question \ref{q: Qi-tw} has a positive answer for  $G\in \mathcal{G}$. As a consequence,  we can still use positive instances of the conjecture to obtain positive results on Question \ref{q: Qi-tw} for various graph classes, as we will see below.

\smallskip

As in the case of  Question \ref{q: minors}, we also ask if Question \ref{q: Qi-tw} admits a positive answer for general infinite graphs (not necessary locally finite), when replacing $\wco$ by its variant $\wco'$ that we will define in Section \ref{sec: alt}.
\medskip

Let $\mathcal{G}$ be an infinite graph class. If a connected locally finite graph $G\in \mathcal{G}$ excludes some finite planar graph $H$ as an asymptotic minor, then there is an integer $k$ such that $G$ excludes the $k\times k$ grid as an asymptotic minor. The validity of  Conjecture \ref{conj:GP} for  $G\in \mathcal{G}$  would then imply that $G$ is quasi-isometric to a graph of bounded treewidth, and in turn Proposition \ref{prop: tw-ub} would imply that $G$ has bounded weak cop number. This shows that the validity of Conjecture \ref{conj:GP} for a given graph class $\mathcal{G}$ also 
implies a positive answer to the first part of Question
\ref{q: minors} for the same class $\mathcal{G}$.



\subsection{Locally finite quasi-transitive graphs and finitely generated groups}
\label{S:virtually-free}

A graph $G$ is \emph{vertex-transitive} if the vertex set of $G$ has a unique orbit under the action of the automorphism group of $G$. The graph $G$ is \emph{quasi-transitive} if the vertex set of $G$ has finitely many orbits under the action of the automorphism group of $G$. Note that every Cayley graph of a group is vertex-transitive, and every vertex-transitive graph is quasi-transitive.

\medskip

In the locally finite quasi-transitive setting, it is well known that the property of having finite treewidth is equivalent to the property of being quasi-isometric to a tree (we give a sketch of the proof below for the sake of completeness).

\begin{theorem}[Folklore]
\label{thm: QT-tw}
 Let $G$ be a locally finite quasi-transitive graph. Then $G$ is quasi-isometric to a tree if and only if $G$ has finite treewidth.
\end{theorem}

\begin{proof}
 The direct implication holds more generally without the quasi-transitivity assumption if we assume that $G$ is of finite maximum degree $\Delta$.
 To see this, let $T$ be a tree and $f: V(G)\to V(T)$ be such that there exists some $C\geq 1$ such that for every $x,y\in V(G)$, we have 
 \[\frac{1}{C}\cdot d_G(x,y)-C\leq d_T(f(x),f(y)) \leq C\cdot d_G(x,y)+C,\]
 and such that for every $t\in V(T)$, there exists $x\in V(G)$ such that $d_T(f(x),t)\leq C$. It is not hard to show that as $G$ has finite bounded degree,  it is also quasi-isometric to some tree of finite bounded degree (see for example \cite[Lemma 4.2]{EG24}), thus we may assume that $T$ has finite bounded degree. We let $T'$ be the tree of finite bounded degree obtained from $T$ by attaching $\Delta^{C^2}+1$ pendant leaves to each node.  
 We claim that $G$ is a subgraph of the graph $(T')^{2C+2}$. For this, we define an injective graph homomorphism $g:G\to (T')^{2C+2}$ as follows. For every $x\in V(G)$, we let $g(x)$ be one of the $\Delta^{C^2}+1$ pendant vertices we attached to the node $f(x)$ in $T'$, in such a way that for each $t\in V(T)$, $g$ maps injectively the vertices of $f^{-1}(t)$ to the pendant vertices attached to $t$. Note that it is always possible to find such a mapping $g$, as for every $t\in V(T)$ and $x,y\in V(G)$ with $f(x)=f(y)=t$, we have $d_G(x,y)\le C^2$ and thus  $|f^{-1}(t)|\leq \Delta^{C^2}+1$. If two vertices $x,y\in V(G)$ are adjacent in $G$, then $d_T(f(x),f(y)) \leq  2C$ and thus $d_{T'}(g(x),g(y)) \leq  2C+2$. This shows that $g$ is indeed an injective homomorphism to $(T')^{2C+2}$, and thus $G$ is a subgraph of $(T')^{2C+2}$.
 As $T'$ is a tree of finite bounded degree, the graph $(T')^{2C+2}$ has finite treewidth, hence $G$ also has finite treewidth. 
 
 We now give a sketch of the proof for the converse implication. Let $G$ be a locally finite quasi-transitive graph of finite treewidth. By Halin's grid theorem \cite{HalinGrid}, all the ends of $G$ are thin, and by a theorem of Woess \cite{Woess89}, they moreover have finite bounded size. A consequence of a result of Carmesin, Hamann and Miraftab \cite[Theorem 7.3]{CTTD}, is the existence of a canonical tree-decomposition $(T,(V_t)_{t\in V(T)})$ of finite bounded adhesion distinguishing the ends of $G$. In particular, this decomposition has the property that $\mathrm{Aut}(G)$ induces a quasi-transitive action on $V(T)$. Hence this implies 
 that the set of values $\{d_{G}(u,v): t\in V(T), u,v\in V_t\}$ admits some upper bound $A\in \mathbb N$. One can then show that any mapping $f: V(G)\to V(T)$ satisfying that for each $x\in V(G), x\in V_{f(x)}$ defines a quasi-isometry.
\end{proof}

An immediate consequence of Corollary \ref{cor: wcn1} and Theorem \ref{thm: QT-tw} is then the following.

\begin{corollary}
 \label{cor: QT-wcn1}
 Let $G$ be a locally finite connected quasi-transitive graph. Then $\wco(G)=1$ if and only if $G$ has finite treewidth. 
\end{corollary}

In \cite[Question K]{lee2023coarse}, the authors asked whether for every finitely generated group $\Gamma$, we have $\wco(\Gamma)=1\text{ or }\infty$. 
By Corollary \ref{cor: QT-wcn1}, a 
positive answer to the first part of Question \ref{q: Qi-tw} would imply that for every locally finite quasi-transitive  graph $G$, $\wco(G)\in \{1,\infty\}$. This follows from the following (folklore) proposition. 

\begin{proposition}[Folklore]
\label{prop: tw-QI}
Let $G$ be a graph of finite maximum degree, and $H$ be a graph of finite treewidth such that $G$ admits a quasi-isometric embedding in $H$. Then $G$ also has finite treewidth. 
\end{proposition}

\begin{proof} 
 We let $\Delta\in \mathbb N$ denote the maximum degree of $G$.
 Let $f: V(G)\to V(H)$ and $C\geq 1$ be such that for every $x,y\in V(G)$, we have 
 \[\frac{1}{C}d_G(x,y)-C\leq d_T(f(x),f(y)) \leq C\cdot d_G(x,y)+C.\]
 We let $(T,\mathcal V)$ be a tree-decomposition of $H$ of finite width $k$, with $\mathcal V= (V_t)_{t\in V(T)}$, and for every $t\in V(T)$, we let $V'_t:=f^{-1}(N_H^{2C}[V_t])$, where $N_H^{2C}[V_t]$ denotes the ball of radius $2C$ around $V_t$ in $H$ (note that we potentially have $V'_t=\emptyset$ if no vertex of $N_H^{2C}[V_t]$ has a preimage under $f$). We claim that $(T,(V'_t)_{t\in V(T)})$ is a tree-decomposition of finite width of $G$.
 
 We first check that $(T,(V'_t)_{t\in V(T)})$ is a tree-decomposition. The first item from the definition of tree-decomposition immediately holds as $(T, \mathcal V)$ is a tree-decomposition. To see that the second item holds, we consider three nodes $t,t',t''\in V(T)$ such that $t'$ is on the unique path between $t$ and $t''$ in $T$. As $(T,\mathcal V)$ is a tree-decomposition, we have $V_{t}\cap V_{t''}\subseteq V_{t'}$. We claim that moreover $N_H^{2C}[V_t]\cap N_H^{2C}[V_{t''}]\subseteq N_H^{2C}[V_t\cap V_{t''}]$, which will be enough to conclude then as $N_H^{2C}[V_{t}\cap V_{t''}]\subseteq  N_H^{2C}[V_{t'}]$ and thus, $V'_{t}\cap V'_{t''}\subseteq V'_{t'}$. To show the inclusion $N_H^{2C}[V_t]\cap N_H^{2C}[V_{t''}]\subseteq N_H^{2C}[V_t\cap V_{t''}]$, we let $x\in N_H^{2C}[V_t]\cap N_H^{2C}[V_{t''}]$, and let $a\in V_t$ and  $b\in V_{t''}$ such that $d_H(a,x)\leq 2C$ and $d_H(b,x)\leq 2C$. In particular, $x$ lies on a path $P$ of length $4C$ from $a$ to $b$ obtained from concatenating two paths $P_1$ and $P_2$, each of length at most $2C$ respectively from $a$ to $x$ and from $x$ to $b$. As $(T,\mathcal V)$ is a tree-decomposition, $V_t\cap V_{t''}$ separates $V_t$ from $V_{t''}$ in $H$, hence $P$ must intersect $V_t\cap V_{t''}$. This implies then that $x\in N_H^{2C}[V_t\cap V_{t''}]$, as claimed above. 
 The third item of the definition of a tree-decomposition follows from the fact that for every edge $xy\in E(G)$, $d_H(f(x),f(y))\leq 2C$, hence for every $t\in V(T)$ such that $f(x)\in V_t$, both $x$ and $y$ must belong to $V'_t$. It remains to check that $(T, (V'_t)_{t\in V(T)})$ has finite width. First, note that for every $z\in V(H)$, $f^{-1}(N_H^{2C}[z])$ has diameter at most $C\cdot(4C+C)$ in $G$. In particular, $|f^{-1}(N_H^{2C}[z])|\leq \Delta^{C\cdot(4C+C)}$, implying that for each $t\in V(T)$, $|V'_t|\leq (k+1)\cdot \Delta^{C\cdot(4C+C)}$.
\end{proof}

Now  consider any locally finite connected quasi-transitive graph $G$ with  $\wco(G)<\infty$. If the first part of Question \ref{q: Qi-tw} (even restricted to locally finite quasi-transitive  graphs) has a positive answer, then this implies that $G$ is quasi-isometric to a graph of bounded treewidth. Since $G$ is quasi-transitive and locally finite, it has bounded maximum degree, and thus by Proposition \ref{prop: tw-QI}, $G$ has bounded treewidth, and then Corollary \ref{cor: QT-wcn1} implies that $\wco(G)=1$. This shows that if the first part of Question \ref{q: Qi-tw} has a positive answer, then for any locally finite connected quasi-transitive graph $G$,  $\wco(G)\in \{1,\infty\}$.

\medskip

We have observed above that  the validity of  Conjecture \ref{conj:GP} for a given graph class $\mathcal{G}$ would directly imply a positive answer to the first part of Question \ref{q: Qi-tw} for the same class $\mathcal{G}$. It follows from the paragraph above that if the conjecture is valid for locally finite quasi-transitive graphs, then it would also imply that for every locally finite quasi-transitive graph $G$, $\wco(G)\in \{1,\infty\}$.

\medskip

In \cite{lee2023coarse}, the authors proved that every finitely generated group which is virtually free has weak cop number $1$, and asked whether the converse implication also holds. 
We now show that this is a simple  consequence of Corollary \ref{cor: QT-wcn1}.

\begin{corollary}
 \label{cor: Group-tw}
 Let $\Gamma$ be a finitely generated group. Then $\wco(\Gamma)=1$ if and only if $\Gamma$ is virtually free.
\end{corollary}

\begin{proof}
 As mentioned above, the converse implication was proved in \cite{lee2023coarse}. Let $\Gamma$ be a finitely generated group such that $\wco(\Gamma)=1$, and let $G=\mathrm{Cay}(\Gamma, S)$ be some of its associated Cayley graphs, for some finite set $S$ of generators of $\Gamma$.
 By Corollary \ref{cor: QT-wcn1}, $G$ has finite treewidth. By Halin's grid theorem \cite{HalinGrid}, all the ends of $G$ must be thin. In particular, a result of Karrass,
 Pietrowski and Solitar \cite{KPS} implies that $\Gamma$ must be virtually free. 
\end{proof}

Albrechtsen and Hamann \cite{AH24} proved that every finitely presented group which is not virtually free admits the infinite square grid as an asymptotic minor. Thus, combining this with Theorem \ref{thm: minor}, we immediately obtain the following corollary, implying a positive answer to \cite[Question K]{lee2023coarse} in the special case of finitely presented groups.

\begin{corollary}
 \label{cor: weak-infty}
 Every finitely presented group which is not virtually free has infinite weak cop number. In particular, for any finitely presented group $\Gamma$, $\wco(\Gamma)\in \{1,\infty\}$.
\end{corollary}

\section{The strong cop number of the 2-dimensional grid}\label{sec:z2}

In this section we prove the following.

\begin{theorem}
$\sco(\mathbb{Z}^2)=\infty$.
\end{theorem}

\begin{proof}
For every integer $t\ge 1$, we give a winning strategy for the robber against $t$ cops in the infinite $2$-dimensional square grid (which is a Cayley graph of $\mathbb Z^2$). Assume the cops are running at speed $s$. The robber will travel at speed $s_r=5 (t+1)(2t+1)s$. Let $\rho$ be the radius of capture of the cops. We assume for simplicity that $\rho\ge 1$ in the computation (we note that if the cops can win with a radius of capture of 0, they can also win with $\rho=1$). 
Consider a subgrid $B$ of size $(2t+1)\cdot 5\rho (t+1)s$ by $5\rho (t+1)s$, and divide it into $2t+1$ subgrids of size $5\rho (t+1)s$ by $5\rho (t+1)s$, which we will call \emph{rooms}. 

Observe that for any two rooms $R$ and $R'$, there are $t+1$ paths $P_1, \ldots, P_{t+1}$ from $R$ to $R'$, such that
\begin{itemize}
    \item any two paths lie at distance at least $5\rho s$ apart, and 
    \item for any $x\in R$  and every $1\le i \le t+1$, there is a path $Q\subseteq R$ such that the concatenation of the two paths $Q, P_i$ is a path between $x$ and $R'$ of length at most $5\rho (t+1)(2t+1)s$.
\end{itemize}
See Figure \ref{fig:grid} for an illustration.

\begin{figure}[htb] 
  \centering 
  \includegraphics[scale=0.8]{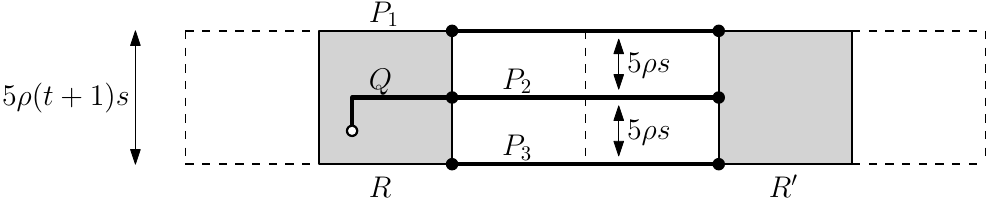}
  \caption{The $t+1$ paths between any two rooms $R$ and $R'$.}
  \label{fig:grid}
\end{figure}

We say that a room $R$ is \emph{safe} if it lies at distance more than $(2\rho+1) s$ from all the cops, and \emph{super safe} if it lies at distance more than $(3\rho +1)s$ from all the cops.  Note that each ball of radius $(3\rho+1)s\le 4\rho s$ in the grid intersects at most 2 rooms, and thus at most $2t$ rooms are not super safe at any time. As there are $2t+1$ rooms, it follows that at any step in the game, at least one room is super safe (and thus safe).

At the beginning of the game, the robber chooses a safe room $R$ and remains there until it is not safe anymore. At the previous round the room $R$ was still safe, so all cops were at distance at least $(2\rho+1)s$ from $R$. It follows that after their move all cops are still at distance at least $(2\rho+1)s-s=2\rho s$ from $R$. We then consider another room $R'$ which is super safe after the moves of the cops, and by the observation above there are $t+1$ paths $P_1,\ldots, P_{t+1}$ from $R$ to $R'$, such that any $P_i,P_j$ lie  at distance at least $5\rho s$ apart. Note that the ball of radius $2\rho s$ around a cop can intersect at most one of these paths, and therefore there is a path $P_i$ between $R$ and $R'$ which lies at distance more than  $2\rho s$ from all the cops. Let $Q$ be the path of $R$ between the current location of the robber to the starting point of $P_i$, such that the concatenation of $Q,P_i$ has length at most $5\rho (t+1)(2t+1)s$.

During $\rho$ steps, the robber travels along $Q,P_i$ from its current location in $R$ to $R'$ (this is possible, since the speed of the robber is $s_r=5 (t+1)(2t+1)s$, while the concatenation of $Q,P_i$ has length at most $5\rho (t+1)(2t+1)s$). During these steps, the cops have only travelled a distance of at most $\rho s$ from their previous locations, and thus they are still at distance more than $2\rho s-\rho s=\rho s\ge \rho$ from $P_i$ during all the moves, unable to catch the robber. Moreover the cops were at distance at least $(3\rho+1)s$ from $R'$, by the definition of a super safe room. It follows that after $\rho$ steps the cops are still at distance $(3\rho+1)s-\rho s=(2\rho+1)s$ from $R'$, so $R'$ remains safe. The robber can thus stay there until it is not safe anymore, and then repeat the strategy above, evading the cops forever while remaining in $B$ the whole time.
\end{proof}

\section{Alternative games}\label{sec:div}

\subsection{The divergence game}

In most of our proofs, when we provide a winning strategy for the robber, he does much more than coming back to a given ball of bounded radius infinitely often: he stays in the ball forever. This motivates the following variant of the game, which we call the \emph{divergence game}. The setting is the same as in the weak and strong versions of the cops and robber game studied in this paper, except that the goal of the cops is now to make sure that the limit inferior of the distance between some fixed vertex $v$ of $G$ and the location of the robber is infinite. Equivalently, the goal of the robber is to stay within some ball of bounded radius forever.

With the strong and weak versions, this defines two new cop numbers, which we call $\wdi(G)$ for the weak version and $\sdi(G)$ for the strong version. Note that this game is harder for the robber and thus for any graph $G$ we have $\wdi(G)\le \wco(G)$ and $\sdi(G)\le \sco(G)$.

\medskip

A mild variant of the proof of Theorem \ref{thm:qi} shows the following.

\begin{theorem}\label{thm:qidiv}
Assume that a connected graph $G$ has a quasi-isometric embedding into a connected graph $H$. Then $\wdi(G)\le \wdi(H)$ and $\sdi(G)\le \sdi(H)$. In particular the two parameters are invariant under quasi-isometry.
\end{theorem}

The proofs of Theorem \ref{thm: hyperbolic} and Corollary \ref{cor: wcn1} (via Theorem \ref{thm: minor}) also show that any graph with $\sdi(G)=1$ is hyperbolic, and that any graph with $\wdi(G)=1$ is quasi-isometric to a tree. On the other hand, by \cite[Theorem 3.2]{lee2023coarse}, every hyperbolic graph $G$ satisfies $\sdi(G)\le \sco(G)\le 1$ and thus $\sdi(G)=1$. Similarly, every graph which is quasi-isometric to a tree satisfies $\wdi(G)\le \wco(G)\le 1$ and thus $\wdi(G)=1$.
This shows the following.

\begin{theorem}\label{thm:div1}
For any graph $G$, $\sdi(G)=1$ if and only if $\sco(G)=1$, and $\wdi(G)=1$ if and only if $\wco(G)=1$.
\end{theorem}

A natural question is whether this holds more generally: is it true  that $\sdi(G)=\sco(G)$ and $\wdi(G)=\wco(G)$ for any graph $G$? Is it true  that $\sdi(\Gamma)=\sco(\Gamma)$ and $\wdi(\Gamma)=\wco(\Gamma)$ for any group $\Gamma$?

\subsection{Alternative versions of the weak and strong games}
\label{sec: alt}

Finally we go back to the connection between the weak version of the original game and treewidth. Recall that 
Proposition \ref{prop: tw-ub} implies that if $G$ is a locally finite connected graph, then $\wco(G)\leq \tw(G)+1$. On the other hand, Proposition \ref{ex: tw2-wcop-infty} provides an example of a graph $G$ with a single vertex of infinite degree, such that $G$ has treewidth 2 but infinite weak cop number. This shows that the assumption that graphs are locally finite is crucial in Proposition \ref{prop: tw-ub}.

\medskip

Consider the following alternative definition of the weak and strong cop numbers of a graph $G$: instead of preventing the robber to come back infinitely often in some ball of finite radius in $G$, we require that that they prevent the robber from coming back infinitely often in some finite subgraph of $G$. Let $\wco'(G)$ and $\sco'(G)$ denote the two corresponding cop numbers. Note that these games are harder for the robber, and thus $\wco'(G)\le \wco(G)$ and $\sco'(G)\le \sco(G)$ for any graph $G$. On the other hand, the games are precisely the same as the original ones in locally finite graphs, where a connected subgraph has bounded radius if and only if it is finite. It follows that for any locally finite graph $G$, $\wco'(G)= \wco(G)$ and $\sco'(G)= \sco(G)$. 

\medskip

It is not difficult to modify the proof of Proposition \ref{prop: tw-ub} to show that if $G$ is a  connected graph (not necessarily locally finite), then $\wco'(G)\leq \tw(G)+1$. Proposition \ref{ex: tw2-wcop-infty} thus provides an example of a graph $G$ with $\wco(G)=\sco(G)=\infty$ but $\sco'(G)\le\wco'(G)\le 3$. 
We also observe that the proof of Theorem \ref{thm: minor} shows that if $H$ is a finite asymptotic minor of a graph $G$, then $\wco'(G)\geq \tw(H)$. Our main motivation to introduce these alternative games is that they allow to ask Questions \ref{q: minors} and \ref{q: Qi-tw} on general infinite graphs (which are not necessarily locally finite). 

\begin{question}
 \label{q: minors'}
 Is it true that if a  connected graph $G$ excludes some finite planar graph $H$ 
 as an asymptotic minor, then
 $\wco'(G)< \infty$?   If so, does there exist some 
 function $f:\mathbb N\to \mathbb N$ such that \[\wco'(G)\leq f\big(\max\sg{\tw(H): |H|<\infty~\text{and}~H\preccurlyeq_{\infty} G}\big)?\] Can we choose $f=\mathrm{id}_{\mathbb N}$? 
\end{question}

\begin{question}
 \label{q: Qi-tw'}
 Is it true that if a graph $G$ is such that $\wco'(G)<\infty$,  then it is quasi-isometric to some graph of finite treewidth? If so, is it quasi-isometric to some graph of treewidth at most $f(\wco'(G))$ for some $f:\mathbb N\to \mathbb N$? Can we choose $f(k)=k+1$? 
\end{question}

Finally, we observe that  the proof of Theorem \ref{thm:qi} can be slightly modified to show the following.

\begin{theorem}\label{thm:qidiv'}
Assume that a connected graph $G$ has a quasi-isometric embedding into a connected graph $H$. Then $\wco'(G)\le \wco'(H)$ and $\sco'(G)\le \sco'(H)$. In particular the two parameters are invariant under quasi-isometry.
\end{theorem}

To see this, observe that when the robber moves in $G$ in some finite subgraph $X$, the image of the robber in $H$ under the quasi-isometry $h$ moves in the image $h(X)$ of $X$  in $H$, and this is a finite subgraph (which can be obtained by adding a finite number of finite paths between the elements of $h(X)$ in $H$).

\subsection*{Acknowledgement} We thank Raphael Appenzeller for the helpful comments, and two anonymous referees for their suggestions.

\bibliographystyle{alpha}
\bibliography{biblio}

\end{document}